\documentclass[a4paper]{amsart}
\usepackage{mathtools}
\usepackage{amsmath,amssymb}

\usepackage{graphicx,hyperref}
\usepackage{etoolbox}
\usepackage{dsfont}

\usepackage{enumitem}

\hypersetup{colorlinks, urlcolor=blue, citecolor=red, linkcolor=blue,
    pdfcreator=LaTeX, breaklinks=true, pdfpagelayout=SinglePage, bookmarksopen=true,bookmarksopenlevel=2}

\DeclarePairedDelimiterX\set[1]\{\}{#1}

\DeclarePairedDelimiterX\dual[2]{\langle}{\rangle}{#1,#2}

\DeclarePairedDelimiter\abs{\lvert}{\rvert}
\DeclarePairedDelimiter\norm{\lVert}{\rVert}

\DeclarePairedDelimiter\parens()

\newcommand\NN{\mathcal{N}}
\newcommand\RR{\mathcal{R}}
\newcommand\TT{\mathcal{T}}

\newcommand\R{\mathbb{R}}

\newcommand{\weakly}{\rightharpoonup}

\newcommand\Uad{U_{\mathrm{ad}}}

\newcommand\cl{\operatorname{cl}}

\renewcommand\d{\mathop{}\!\mathrm{d}}

\newtheorem{theorem}{Theorem}[section]
\newtheorem{lemma}[theorem]{Lemma}
\newtheorem{corollary}[theorem]{Corollary}
\newtheorem{remark}[theorem]{Remark}
\newtheorem{definition}[theorem]{Definition}
\newtheorem{example}[theorem]{Example}
\newtheorem{proposition}[theorem]{Proposition}

\numberwithin{equation}{section}

\begin{document}
\title[No-gap second-order conditions for additive manufacturing]{No-gap second-order optimality conditions for additive manufacturing}
\author{Hiba Hmede}
\author{Luc Paquet}
\author{Gerd Wachsmuth}

\address{Univ. Polytechnique  Hauts-de-France, INSA Hauts-de-France,  CERAMATHS Laboratoire de Mat\'eriaux C\'eramiques et de Math\'ematiques, F-59313 Valenciennes, France
and
Lebanese University, Faculty of sciences 1, Khawarizmi Laboratory of Mathematics and Applications-KALMA, Hadath, Beirut, Lebanon}
\address{Univ. Polytechnique  Hauts-de-France, INSA Hauts-de-France,  CERAMATHS Laboratoire de Mat\'eriaux C\'eramiques et de Math\'ematiques, F-59313 Valenciennes, France}
\address{Brandenburgische Technische Universit\"at Cottbus-Senftenberg, Institute of Mathematics, 03046 Cottbus, Germany}

\begin{abstract}
Additive manufacturing by laser fusion on a metal oxides powder bed has developed considerably in the last 
few years and allows to produce a wide range of complex parts. The mathematical 
models correspond to initial boundary value problems for the heat equation with moving 
heat sources according to the laser trajectories. The main questions concern the 
optimization of the trajectories scanned by the laser and of the thermal treatment time in order to melt the powder where it is desired to make the part and to minimize the thermal gradients. Our purpose in this current paper is to pursue the study of the optimization model that we have introduced in a previous paper. Here, we consider second-order optimality conditions for non-necessarily convex constraints on the laser paths. In particular, we obtain no gap between the second-order sufficient optimality condition and the necessary second-order optimality condition. To achieve this goal, we reformulate our optimal control problem in order to fit it in the framework of the abstract theory of optimization under constraints in Banach spaces. Higher regularity of the trajectories for local minimizers is also proved implying higher regularity of the corresponding Lagrange multipliers. The case of the regularity of the trajectories for stationary points is left open.
\end{abstract}
\maketitle

\section{\textbf{Introduction}}
\label{sec:intro}
Additive manufacturing is a method used to produce 3D objects from metal oxides
	powder, layer by layer, by applying  a Gaussian laser beam to each slice as a heat source to fuse
	together the powder.
	However, the main problem is that repeated and rapid heating and cooling cycles of the metal oxide powder during this {process} cause high temperature gradients that contribute to inner stresses and sometimes result in cracks in the fabricated part.
	
	As a possible answer to this problem, we have introduced  and studied in \cite{HNPW} a laser path optimization model aiming to design laser paths such that, during the thermal process, the temperature reaches a melting value across the structure to be built, while minimizing thermal gradients. In this model we also seek to adjust the temperature distribution in two time phases of length $r$ each, within the shortest global time $\tau$. The necessity of a preheating phase during the process of selective laser melting has been motivated by \cite{preheat} where the importance of this first phase is highlighted to reduce cracks and achieve higher density. Requirement about the shortest treatment time has been considered in the context of tumor treatment by cytotoxic drugs in \cite{Alsay}. Related problems but with fewer requirements have been considered in \cite{TMA,AlNiPa,MBois}. 
	In its full generality, we assume   $\Omega\subset\mathbb{R}^3$ {to} be a bounded Lipschitz domain with boundary $\Gamma$ decomposed into $\Gamma_i, i=1,2,3$, disjoint open subsets of $\Gamma$ such that {${\textstyle\bigcup_{i=1}^{3}}
	\ \bar{\Gamma}_{i}=\Gamma$}. We assume that $\Gamma_1$ is contained in a plane of $\mathbb{R}^3$ that we identify with $\mathbb{R}^2$ and that for every $x=(x_1,x_2,x_3) \in \Omega$, such that $(x_1,x_2,0)\equiv (x_1,x_2) \in \Gamma_1$,  $x_3 <0$. We consider $\Gamma_S$ a fixed non-empty compact subset of $\Gamma_1$ and $T>0$ a fixed maximal global time for the laser thermal process.
In practical applications, $\Omega=\Gamma_1\times(-h,0),$ is the powder bed layer, where $h>0$ is its thickness,  $\Gamma_1\subset \mathbb{R}^2$ is its upper surface, $\Gamma_3$ is the bottom of the powder layer i.e. the top of the previous layer or the building platform, and $\Gamma_2$ its lateral surface (see Figure \ref{fig1}). The industrial goal is to solidify  $\Gamma_S\times(-h,0),$ by using a laser path $\gamma:t\in [0,T]\mapsto \gamma(t)\in \Gamma_S$ (physically $\gamma$ means the trajectory of the laser spot  on $\Gamma_1$). 

\begin{figure}[htp]
	\centering
	\includegraphics[scale=0.5]{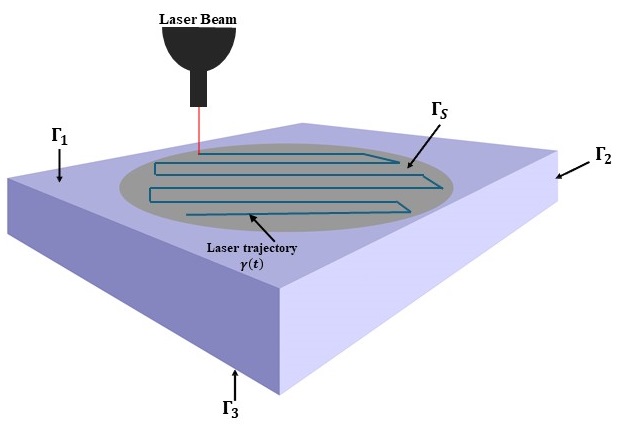}\caption{The powder bed layer.}
	\label{fig1}
\end{figure}

  To reach all these purposes, the following optimal control problem has been introduced in \cite{HNPW}:
Minimize the cost functional
		$J_r:W(0,T)\times H^1(0,T;\R^2)\times[3r,T]\rightarrow \R $ defined by
	\begin{equation}\label{eq:costfunctional}
		\begin{aligned}
			J_r (y,\gamma,\tau)
			&=
			\frac{\lambda_Q}{2}\frac{1}{r}\int_{\frac{\tau}{2}-\frac{r}{2}}^{\frac{\tau}{2}+\frac{r}{2}}\int_{\Omega}{|y(x,t)-y_Q(x,t)|^2 \d x \d t }
			\\&\qquad
			+\frac{\lambda_\Omega}{2}\frac{1}{r}\int_{\tau-r}^{\tau}\int_{\Omega}{|y(x,t)-y_\Omega(x,t)|^2 \d x \d t }
			\\&\qquad +\frac{1}{2}\int_0^T\int_\Omega |\nabla y(x,t)|^2 \d x \d t  +\frac{\lambda_\gamma}{2}
			|\gamma| ^2_{H^1(0,T;\mathbb{R}^2)}+\beta_T\tau
		\end{aligned}
	\end{equation}
subject to
\begin{itemize}
\item the parabolic linear initial boundary value problem with Robin boundary conditions:
\begin{equation}\label{eqn:linsystem}
	\begin{cases}
	\rho c \partial_t y - \kappa \Delta y = 0 \qquad\qquad & \text{in }  Q=\Omega\times]0,T[,\\
	-\kappa \partial_\nu y = hy-g_\gamma  &\text{on }  \Sigma_1:=\Gamma_1\times]0,T[, \\
	-\kappa \partial_\nu y = hy  & \text{in }   \Sigma_2:=\Gamma_2\times]0,T[, \\
	-\kappa \partial_\nu y = h(y-y_B)   &   \text{in }  \Sigma_3=\Gamma_3\times]0,T[,  \\
		y(x,0) = y_0(x) &  \text{for }  x\in \Omega,
	\end{cases}
\end{equation} 
where $g_{\gamma}$ is the source of heat (see \eqref{iglb}) corresponding to the laser path $\gamma$ on $\Gamma_1$, $y_0$ is a given initial condition  and $y_B $  a given distribution of temperature on $\Gamma_3$; 
\item	  
$\gamma$ to belong to the set of admissible paths $\Uad$ for the laser path on $\Gamma_1$:
\begin{equation}
	\label{eq:setadmcdts}
	\Uad := \set{ \gamma \in H^1(0,T;\mathbb{R}^2);\ 
	\gamma(t) \in \Gamma_S \text{, for all } t\in[0,T]},
\end{equation}
\end{itemize}
	where
\begin{equation}
	\label{eq:seminorm}
	\abs{\gamma}_{H^1(0,T; \mathbb{R}^2)}:
	=
	\parens*{
		\int_0^T \abs{\dot{\gamma}(t)}^2 d t}^{1/2}
\end{equation}
is the seminorm in $H^1(0,T;\R^2)$,
$r \in ]0,\frac{T}{3}[$,
$\lambda_Q,\ \lambda_\Omega$ and $\lambda_\gamma$ are positive constants, \ $\beta_T$ is a  
non-negative constant (thus allowed to be 0), $y_Q$ and $y_{\Omega}$ are target functions.

The control constraint set $\Uad$ takes into account the constraint on the laser range: $R(\gamma):=\gamma([0,T]) \subset \Gamma_S$. This constraint physically describes that $\gamma$ must ensure that the region \textquotedblleft below\textquotedblright \ $\Gamma_1 \setminus \Gamma_S $ is not affected by the fusion process.

In \cite{HNPW}, we have shown the existence of an optimal control $(\gamma,\tau)$, proved the first order and second order differentiability of the state mapping as well as of the reduced cost functional $\hat{J}_r$ associated to \eqref{eq:costfunctional}-\eqref{eqn:linsystem}, introduced the adjoint problem, and derived the first order necessary optimality condition. We also have derived a second order sufficient optimality condition but only when $\Gamma_S$ is convex.

It is our purpose in the current paper to pursue the study of that laser path optimization model introduced in \cite{HNPW}. Here, we consider necessary and sufficient second order optimality conditions without the restriction of $\Gamma_S$ being convex, but requiring its boundary to be smooth. We obtain no gap between the second order sufficient optimality condition and the necessary second order optimality condition. To achieve this goal, we reformulate our optimal control problem in order to apply the abstract theory of optimization in Banach spaces, see \eqref{eq:setadmcdtsb}, \eqref{genminpb}, \eqref{refopbeq1}, \eqref{refopbeq2}, \eqref{ourminpb}. Higher regularity of the trajectories for local minimizers of $\hat{J}_r$ is also proved implying higher regularity of the corresponding Lagrange multipliers. The case of the regularity of the trajectories for stationary points is left open.

The paper is structured as follows. In section \ref{sec:prelim}, we precisely state our optimal control problem with its standing assumptions and recall results that we have obtained in  \cite{HNPW} and \cite{AlNiPa} about the state mapping \eqref{CotoStaop}, the adjoint system and the first order necessary conditions. Then, we recall the results concerning the second order continuous Fréchet differentiability of the state operator as well as of the reduced cost functional $\hat{J}_r$ proved in \cite{HNPW}. The $H^2(0,T;\mathbb{R}^2)$-regularity of the trajectories $\bar{\gamma}$ for local minimizers $(\bar{\gamma},\bar{\tau})$ of our reduced cost functional $\hat{J}_r$ on the set $\Uad \times [3r,T]$ is established in section \ref{sec:regularity} and it will be shown in section \ref{sec:soc} that this implies the Lagrange multiplier $\bar{\lambda}$ to belong to $L^2([0,T])$ in this case. In section \ref{subsec:soc}, we recall the basis of the general theory of optimization with constraints in Banach spaces \cite[pp.29-33]{GW}, that we will apply in section \ref{sec:sec_deriv_Leg_form} to our optimal control problem  \eqref{ourminpb}, which is an equivalent reformulation of \eqref{eq:costfunctional}-\eqref{eqn:linsystem}-\eqref{eq:setadmcdts}\}, in order to obtain second order optimality conditions.

In section \ref{sec:soc}, supposing that the function $g$ defining $\Gamma_S$  \eqref{eq:setadmcdtsb} belongs to $C^{3}(\mathbb{R}^{2})$, we prove that the mapping $G:H^{1}(0,T;\mathbb{R}^{2})\rightarrow H^{1}(0,T):\gamma\mapsto g\circ\gamma$ is two times
continuously Fr\'{e}chet differentiable from $H^{1}(0,T;\mathbb{R}^{2})$ into $H^{1}(0,T)$. Then, we prove that the strict Robinson-Zowe-Kurcyusz condition is satisfied at any feasible point $(\bar{\gamma},\bar{\tau})$ for our optimal control problem \eqref{eq:costfunctional}-\eqref{eqn:linsystem}-\eqref{eq:setadmcdts} recast as \eqref{ourminpb} in the general framework of the theory of optimization with constraints in Banach spaces \cite[p.29]{GW} establishing consequently the existence and uniqueness of a Lagrange multiplier $(\bar\lambda,\bar\mu) $ at any $(\bar \gamma,\bar\tau)$ verifying the first order necessary conditions. We prove that $\bar\lambda$ is a positive regular Borel measure on the time interval $[0,T]$.

In section \ref{sec:sec_deriv_Leg_form}, we prove that if $(\bar \gamma,\bar\tau)$ is a stationary point with Lagrange multiplier $(\bar\lambda,\bar\mu) $, that the quadratic form associated to the second-derivative of the Lagrangian
\begin{equation*}
D^2_{(\gamma,\tau)} L(\bar \gamma,\bar\tau,\bar\lambda,\bar\mu): H^1(0,T;\mathbb{R}^2) \times \mathbb{R} \rightarrow \mathbb{R}: (\delta\gamma,\delta\tau) \mapsto D^2_{(\gamma,\tau)} L(\bar \gamma,\bar\tau,\bar\lambda,\bar\mu)[\delta\gamma,\delta\tau]^2
\end{equation*}
is a Legendre form, and we deduce a sufficient condition for $(\bar \gamma,\bar\tau)$ to be a local minimizer of the reduced cost functional $\hat{J}_r$ without gap with respect to the necessary condition. 

We close our paper by mentioning an open problem in the appendix.       

\section{\textbf{Problem setting and preliminaries}}
\label{sec:prelim}
We assume that the compact subset $\Gamma_S \subset \Gamma_1$ can be described by a function $g \colon \R^2 \to \R$
satisfying the conditions
\begin{enumerate} 
	\item
		$g$ is $C^3$,
	\item
		$\Gamma_S = \set{x \in \R^2 ; g(x) \leq 0}$,
	\item
		$g'(x) \ne 0$ for all $x \in \R^2$ with $g(x) = 0.$
\end{enumerate}
Consequently, $\Gamma_S$ is a differentiable submanifold with boundary of class $C^3$ of $\mathbb{R}^2$. In this setting, we can rewrite $\Uad$, given in \eqref{eq:setadmcdts}, as
\begin{equation}\label{eq:setadmcdtsb} 
\begin{array}[c]{c}
\Uad =\set{ \gamma \in H^1(0,T;\mathbb{R}^2);\ 
		(g \circ\gamma)(t)\leq 0 \text{, for all } t\in[0,T]}.
\end{array}
\end{equation}

\subsection{\textbf{Problem setting and standing assumptions}}
\label{subsec:asm}
Let $J_r$ from $W(0,T)\times H^1(0,T;\mathbb{R}^2)\times[3r,T]$ to $\R$ be defined as in \eqref{eq:costfunctional} and \begin{equation}
	W(0,T):=\left\{u\in L^2(0,T;H^1(\Omega));\frac{\d u}{\d t}\in L^2(0,T;H^1(\Omega)^{\star})\right\}.
	\label{deW0T}
\end{equation}
Our optimal control problem associated with the initial boundary value problem \eqref{eqn:linsystem} is formulated as
\begin{equation}
	\left\{\begin{array}{ll}
	 \label{OCP}
\text{minimize} \  J_r(y_{\gamma},\gamma,\tau)  \\ 
\gamma\in \Uad, \tau\in [3r,T]
    \end{array}\right.
\end{equation}
where $y_{\gamma} \in W(0,T)$ denotes the weak solution of the initial boundary value problem \eqref{eqn:linsystem} in the sense of \eqref{Varpb2} below, representing the temperature distribution in $\Omega$ during the time interval $[0,T]$ corresponding to the heat source $g_{\gamma}$ given in \eqref{iglb}.
For this problem, we postulate the following assumptions.
\begin{itemize}
	\item In the definition of the initial boundary value problem \eqref{eqn:linsystem}, $\rho, c, \kappa$ and $h$ are all positive constants. The initial and boundary data satisfies $y_0 \in L^2(\Omega)$ and $y_B \in L^2(\Sigma_3)$. In practical applications, $y_B$ denotes the distribution of temperatures on the top of the previous layer or on the building platform. The fixed final time is given by $T>0$.
	\item
		The heat source induced by the Gaussian laser beam $\gamma$ on $\Gamma_1$ is given by
		\begin{equation} \label{iglb}
			g_{\gamma}(x,t)=\alpha\frac{2P}{\pi R^{2}}\exp\left(-2\frac{\left\vert
			x-\gamma(t)\right\vert ^{2}}{R^{2}}\right) \qquad \text{for all }(x,t)\in\Gamma_{1}%
		\times]0,T[,		  
		\end{equation}
		where the positive constant $P$ denotes the power of the Gaussian laser beam, the positive constant $R$ its radius and the positive constant $\alpha$ the absorption coefficient of the powder. Let us note that \eqref{iglb} is meaningful for any $\gamma \in H^1(0,T;\mathbb{R}^2)$.

	\item In the definition of the cost functional \eqref{eq:costfunctional}, $r\in]0,\frac{T}{3}[$ is a fixed real number representing the length of each of the two time intervals in the heating process considered in \eqref{eq:costfunctional}. Further $\lambda_{Q},\lambda_{\Omega},\lambda_{\gamma}$ are positive constants and $\beta_{T}$ a non-negative constant.
		The desired temperatures enjoy the regularity $y_Q, y_\Omega \in C([0,T];L^2(\Omega))$.
\end{itemize}

\begin{remark}
In the definition \eqref{eq:setadmcdtsb} of the set of the admissible paths $\Uad$,
we have removed the condition that $|\dot{\gamma}|$ is uniformly bounded from the corresponding definition given in \cite{HNPW},
but we assume that the coefficient $\lambda_{\gamma}$ in the cost functional \eqref{eq:costfunctional} is positive.
Consequently, the results of \cite{HNPW} remain valid in our actual setting.

Let us also note that $|\gamma(t)|_{\mathbb{R}^2}$ is uniformly bounded for $\gamma \in \Uad$ and $t \in [0,T]$ due to the fact that the range $R(\gamma)$ of $\gamma$ is contained in the compact set $\Gamma_S$. This allows us to replace $\lVert \gamma \rVert_{H^1(0,T;\mathbb{R}^2)}^2$ in the definition of the cost functional $J_r$ by the squared semi-norm $| \gamma |_{H^1(0,T;\mathbb{R}^2)}^2$.
Again, this is a small difference with respect to the setting in \cite{HNPW}.    
\end{remark}

\subsection{\textbf{Results concerning the state mapping}}
\label{subsec:PDE}
We say that $y \in W(0,T)$ is a weak solution of the heat equation \eqref{eqn:linsystem}, if and only if
\begin{equation}\label{Varpb2}	
	\begin{aligned}
		\MoveEqLeft
		\rho c \langle \frac{\d y}{\d t},\xi \rangle _{H^1(\Omega)^{\star},H^1(\Omega)} + \kappa \int_{\Omega}\nabla y\cdot\nabla\xi   \d x+h\int_{\Gamma}y\cdot\xi \d S(x)
		\\&
		=\int_{\Gamma_1}g_\gamma\cdot\xi \d S(x)+h\int_{\Gamma_3}y_B\cdot\xi \d S(x),
	\end{aligned}
\end{equation} 
holds in $L^{2}([0,T])$
for all $\xi \in H^1(\Omega)$ 
together with the initial condition $y(0)=y_0\in L^{2}(\Omega)$, which makes sense as $W(0,T)\hookrightarrow C([0,T];L^2(\Omega))$\footnote{Here and below, $X\hookrightarrow Y$ means that $X$ is continuously embedded into $Y$.}.
\begin{proposition}
	Under our general assumptions, for every given initial datum $y_{0} \in L^{2}(\Omega)$ and every $\gamma \in H^1(0,T;\mathbb{R}^2)$, problem \eqref{eqn:linsystem} has a unique weak solution $y_{\gamma}$ (we say that $y_{\gamma}$ is the state corresponding to $\gamma$).
\end{proposition}
\begin{proof}
	The existence and uniqueness of a weak solution to \eqref{eqn:linsystem} with initial condition $y(0)=y_0 \in L^{2}(\Omega)$, follows easily from the general results for linear I.B.V.P for parabolic equations \cite[pp.512-513]{9} or \cite[Ch.11]{Chi} applied to the variational formulation \eqref{Varpb2}. 
\end{proof}
\begin{proposition}
	The feasible set $	\Uad$ is a weakly closed subset of $H^1(0,T;\mathbb{R}^2).$ 
\end{proposition}
\begin{proof}
	Let $(\gamma_n)_n$  be a sequence in $\Uad$ weakly convergent to some $\gamma$ in $H^1(0,T;\mathbb{R}^2)$. Then, given the compact embedding from $H^1(0,T;\mathbb{R}^2)$ into $C([0,T];\mathbb{R}^2)$, it implies that $(\gamma_n)_n$ is strongly convergent in $C([0,T];\R^2).$
	For the corresponding ranges, we have $R(\gamma_n)\subset\Gamma_{S}$ and therefore $R(\gamma)\subset\Gamma_{S}$  since  $\Gamma_{S}$ is compact, thus closed.
\end{proof}

We define the path-to-state mapping $S$ associated with the state equation \eqref{eqn:linsystem} via
\begin{equation} \label{CotoStaop}
	S:H^{1}(0,T;\mathbb{R}^{2})\rightarrow W(0,T):\gamma\mapsto S(\gamma):=y_{\gamma}.
\end{equation}

\begin{remark}
	\begin{itemize}
		\item We will also sometimes use the notation $y(\gamma)$ instead of $y_{\gamma}$, if more convenient.
		\item In \cite{HNPW}, the path-to-state mapping is denoted $G$ but to avoid confusion with \eqref{defopG} below, we have denoted it by $S$ in this paper.
	\end{itemize}
\end{remark}

\begin{proposition}\label{GWSC}
	The path-to-state map $S	\colon H^{1}(0,T;\mathbb{R}^{2})\to W(0,T)$ is sequentially continuous from $H^{1}(0,T;\mathbb{R}^{2})$ endowed with the weak topology into $W(0,T)$ endowed with the strong topology.
\end{proposition}
\begin{proof}
The sequential continuity from $\Uad$ endowed with the weak topology into $W(0,T)$ endowed with the strong topology follows from the compact embedding of $H^1(0,T;\mathbb{R}^2)$ into  $C([0,T];\mathbb{R}^2)$, from the continuity of the mapping $\gamma \mapsto g_{\gamma}$ \eqref{iglb} from $C([0,T];\mathbb{R}^2)$ into $L^2(\Sigma_1)$, and from the continuity of the solution of the parabolic linear initial boundary value problem with Robin boundary conditions \eqref{eqn:linsystem} with respect to its data
 \cite[Theorem 3.13, p.150]{FrTr}. 
\end{proof}

As there is no difference between our path-to-state mapping  $S$ with that one considered in \cite{AlNiPa}, we also have the following result about $S$.
\begin{theorem} \cite{AlNiPa} \label{FDG}
	The path-to-state mapping   $S:H^1(0,T;\R^2) \longrightarrow W(0,T)$ is a continuously Fréchet differentiable mapping, and for all $\gamma,\ \delta\gamma\in H^1(0,T;\R^2)$
	the Fréchet derivative is given by
	\begin{equation} \label{diffOfG}
		DS(\gamma)\cdot\delta \gamma =\frac{8\alpha P}{\pi R^4}\  \sigma\big(e^{w(\gamma)}\tilde{\gamma}(\gamma)\cdot\delta\gamma\big),
	\end{equation}
	where $\tilde\gamma(\gamma)(x,t):=x-\gamma(t)$
	for all $(x,t)\in \bar\Gamma_1\times[0,T]$,
	$w(\gamma) := -\frac{2}{R^2}|\tilde\gamma(\gamma)|^2$
	and $\sigma$ is the linear continuous mapping defined by
	\begin{equation} \label{sigop}
		\begin{aligned}
			\sigma \colon L^2(\Sigma_1) \to W(0,T)\colon	\delta v  \mapsto \delta y_2,
	\end{aligned} \end{equation}
	where $\delta y_2$ is the unique weak solution of 
	\begin{equation}\label{ivbp}
		\begin{cases}
			\rho c \partial_t \delta y_2 - \kappa \Delta \delta y_2 = 0  \quad \ \   &\text{ in }   Q,\\
			\kappa \partial_\nu \delta y_2 + h\delta y_2=\delta v \quad   &\text{ on }   \Sigma_1, \\
			\kappa \partial_\nu \delta y_2 + h\delta y_2=0 \quad\quad \ \ \ &  \text{ on }   \Sigma_2 \cup \Sigma_3,\\
			\delta y_2(\cdot,0) = 0 \qquad  &\text{ in  }\  \Omega.
		\end{cases}
	\end{equation}
\end{theorem}
\subsection{\textbf{Results obtained previously concerning our optimal control problem}}
\label{subsec:oc}

We recall hereafter definitions and results from \cite{HNPW}.
\begin{definition} \label{defrcf} Given a fixed $r \in ]0,T/3[$, we define the reduced cost functional
	\begin{equation}\label{costfunctional}
		\hat J_r	\colon
		H^1(0,T;\mathbb{R}^2)\times[3r,T]  \to \mathbb{R}\colon		( \gamma ,\tau )  \mapsto  \hat J_r(\gamma,\tau)=J_r(S(\gamma),\gamma,\tau),
	\end{equation} 
	where 
	$J_r$ was given in \eqref{eq:costfunctional}.
\end{definition}
Using the reduced cost functional, we can reformulate \eqref{OCP} as
\begin{equation}
	\label{ROCP}
	\left\{
		\begin{aligned}
			&\text{minimize} \  \hat J_r(\gamma,\tau)  \\ 
			&\gamma\in \Uad, \tau\in [3r,T].
		\end{aligned}
	\right.
\end{equation}
\begin{theorem}\label{EGMin} \cite{HNPW} {(Existence of a global minimizer)}\\
	There exists $(\bar \gamma, \bar \tau)\in \Uad\times[3r,T]$ such that $\hat J_r(\bar \gamma, \bar \tau)\leq \hat J_r(\gamma,\tau)$ for every  $(\gamma,\tau)$ in $\Uad\times[3r,T],$ i.e. the optimal control problem \eqref{ROCP} admits at least one global minimizer.
\end{theorem}
\begin{proof}
We give a different proof than in \cite{HNPW} based on our new Proposition \ref{GWSC}.
First, we can observe that 	$\hat J_r$  is non-negative in $ \textstyle{\Uad\times[3r,T]}$, hence $\textstyle{\hat J_r}$ is bounded from below and admits a finite infimum $\textstyle{L}$
	\begin{equation}\label{Ldef}
\		{L:= \inf_{(\gamma,\tau) \in {W_{ad}} } \hat J_r(\gamma,\tau)}. \end{equation}
	Thus, let us consider a minimizing sequence $\textstyle{\{(\gamma_n,\tau_n)\}_{n\in\mathbb{N}}}$ in $\textstyle{\Uad \times [3r,T]}$ such that
	\begin{equation}\lim_{n\xrightarrow{} +\infty } {\hat J_r(\gamma_n,\tau_n)= L
		}\geq 0.	\end{equation}
		Since $\lambda_\gamma > 0$ and since $\Gamma_S$ is bounded,
		the sequence $\textstyle{({\gamma_n})_{n \in \mathbb{N}}}$ is bounded in $H^1(0,T;\mathbb{R}^2)$, so that there exists a subsequence $\textstyle{(\gamma_{n_j})_{j\in\mathbb{N}}}$ such that $\gamma_{n_j} \xrightharpoonup[j \xrightarrow{} \infty]
{ }\bar\gamma$ in $H^1(0,T;\mathbb{R}^2)$. Consequently by \cite[Theorem 2.12]{FrTr} we get  $ \Arrowvert\bar\gamma\Arrowvert_{H^1(0,T;\mathbb{R}^2)}^2\leq \liminf_{j\to +\infty}\Arrowvert\gamma_{n_j}\Arrowvert_{H^1(0,T;\mathbb{R}^2)}^2$. 
	The sequence $\textstyle{\left(\tau_{n_j}\right)_{j\in\mathbb{N}}}$ is such that $\textstyle{3r \leq\tau_{n_j}\leq T}$, so that $\textstyle{(\tau_{n_j})_{j\in \mathbb{N}}}$ is bounded in $\mathbb{R}$. Therefore, there exists a subsequence $\textstyle{(\tau_{n_{j_{k}}})_{k\in\mathbb{N}}}$ that we rename $\textstyle{(\tau_{n_j})_{j \in \mathbb{N}}}$, and $\textstyle{\bar \tau \in [3r,T]}$, such that 
	$\lim_{j\rightarrow +\infty}\tau_{n_j} = \bar \tau \in [3r,T]$. By Proposition \ref{GWSC}, $S(\gamma_{n_j}) $ converges strongly to $S(\bar\gamma)$ in $ W(0,T)$. 
	Thus, in particular $S(\gamma_{n_j}) \xrightarrow[j \xrightarrow{} +\infty]{}S(\bar\gamma)$  in  $L^2(0,T;H^1(\Omega))$.  	
		
		 Consequently
\begin{equation} \label{CvGrad}
		 \frac{1}{2}\int_0^T\int_\Omega |\nabla S(\gamma_{n_j})|^2 \d x \d t  \to \frac{1}{2}\int_0^T\int_\Omega |\nabla S(\bar{\gamma})|^2 \d x \d t  \text{  as } j \to +\infty .
\end{equation}

To conclude that 
\begin{equation} \label{NV0}
L=\lim_{j \to +\infty}\hat{J}_r(\bar \gamma_{n_j},\bar \tau_{n_j})=\liminf_{j \to +\infty}\hat{J}_r(\bar \gamma_{n_j},\bar \tau_{n_j})\geq \hat{J}_r(\bar \gamma,\bar \tau),
\end{equation}
it remains to prove that 	
\begin{equation} \label{NV1}
\int_{\frac{\bar\tau}{2}-\frac{r}{2}}^{\frac{\bar\tau}{2}+\frac{r}{2}}\int_{\Omega}\arrowvert S(\bar\gamma)-y_Q\arrowvert^2 \d x \d t  = \lim_{j \to +\infty}	\int_{\frac{\tau_{n_j}}{2}-\frac{r}{2}}^{\frac{\tau_{n_j}}{2}+\frac{r}{2}}\int_{\Omega}\arrowvert S(\gamma_{n_j})-y_Q\arrowvert^2 \d x \d t ,
\end{equation}
and 
\begin{equation} \label{NV2}
		\int_{\bar\tau-r}^{\bar\tau}\int_{\Omega}\arrowvert S(\bar\gamma)-y_\Omega\arrowvert^2 \d x \d t  = \lim_{j \to +\infty}	\int_{\tau_{n_j}-r}^{\tau_{n_j}}\int_{\Omega}\arrowvert S(\gamma_{n_j})-y_\Omega\arrowvert^2 \d x \d t .
\end{equation}
We have
\begin{equation} \label{NV3}
	\begin{aligned}
		\MoveEqLeft
		\int_{\frac{\bar{\tau}}{2}-\frac{r}{2}}^{\frac{\bar{\tau}}{2}+\frac{r}{2}}%
		\int_{\Omega}|S(\bar{\gamma})-y_{Q}|^{2} \d x \d t-\int_{\frac{\tau_{n_{j}}}%
		{2}-\frac{r}{2}}^{\frac{\tau_{n_{j}}}{2}+\frac{r}{2}}\int_{\Omega}%
		|S(\gamma_{n_{j}})-y_{Q}|^{2} \d x \d t\\
		&=\int_{\frac{\bar{\tau}}{2}-\frac{r}{2}}^{\frac{\bar{\tau}}{2}+\frac{r}{2}%
		}\int_{\Omega}|S(\bar{\gamma})-y_{Q}|^{2} \d x \d t-\int_{\frac{\tau_{n_{j}}}%
		{2}-\frac{r}{2}}^{\frac{\tau_{n_{j}}}{2}+\frac{r}{2}}\int_{\Omega}%
		|S(\bar{\gamma})-y_{Q}|^{2} \d x \d t\\
		&\qquad+\int_{\frac{\tau_{n_{j}}}{2}-\frac{r}{2}}^{\frac{\tau_{n_{j}}}{2}+\frac{r}%
		{2}}\int_{\Omega}\Big[|S(\bar{\gamma})-y_{Q}|^{2}-|S(\gamma_{n_{j}})-y_{Q}%
		|^{2}\Big] \d x \d t.
	\end{aligned}
\end{equation}
The term on the third line in \eqref{NV3} tends to zero as $j \to +\infty$, since
\begin{align*}
	\MoveEqLeft
\Big|\int_{\frac{\tau_{n_{j}}}{2}-\frac{r}{2}}^{\frac{\tau_{n_{j}}}{2}+\frac{r}%
{2}}\int_{\Omega}\big[|S(\bar{\gamma})-y_{Q}|^{2}-|S(\gamma_{n_{j}})-y_{Q}%
|^{2}\big] \d x \d t\Big|
\\&\leq\int_{0}^{T}\int_{\Omega}\Big||S(\bar{\gamma})-y_{Q}%
|^{2}-|S(\gamma_{n_{j}})-y_{Q}|^{2}\Big| \d x \d t\\
&={\textstyle\iint\limits_{Q}}
\Big||S(\bar{\gamma})-y_{Q}|-|S(\gamma_{n_{j}})-y_{Q}|\Big|\ \big[|S(\bar{\gamma}%
)-y_{Q}|+|S(\gamma_{n_{j}})-y_{Q}|\big] \d x \d t\\
&\leq{\textstyle\iint\limits_{Q}}
|S(\bar{\gamma})-S(\gamma_{n_{j}})|\ \big[|S(\bar{\gamma})-y_{Q}|+|S(\gamma
_{n_{j}})-y_{Q}|\big] \d x \d t\\
&\leq
\left\Vert S(\bar{\gamma})-S(\gamma_{n_{j}})\right\Vert _{L^{2} (Q)}
\parens*{
\left\Vert S(\bar{\gamma})-y_{Q}\right\Vert _{L^{2}(Q)}
+
\left\Vert S(\gamma_{n_{j}})-y_{Q}\right\Vert _{L^{2}(Q)}
}
\\
&\rightarrow0\text{ as }j\rightarrow+\infty,
\end{align*}
because $S(\gamma_{n_{j}})\to S(\bar\gamma)$ in $L^2(Q)$. 

By Lebesgue's dominated convergence theorem
\begin{align*}
\int_{\frac{\tau_{n_{j}}}{2}-\frac{r}{2}}^{\frac{\tau_{n_{j}}}{2}+\frac{r}{2}%
}\int_{\Omega}|S(\bar{\gamma})-y_{Q}|^{2} \d x \d t&=\int_{0}^{T}\int_{\Omega
}\mathds{1}_{\left[  \frac{\tau_{n_{j}}}{2}-\frac{r}{2},\frac{\tau_{n_{j}}}%
{2}+\frac{r}{2}\right]  }(t)\ |S(\bar{\gamma})-y_{Q}|^{2}(x,t) \d x \d t\\
&\rightarrow\int_{0}^{T}\int_{\Omega}\mathds{1}_{\left[  \frac{\bar{\tau}}%
{2}-\frac{r}{2},\frac{\bar{\tau}}{2}+\frac{r}{2}\right]  }(t)\ |S(\bar{\gamma
})-y_{Q}|^{2}(x,t) \d x \d t
\end{align*}
so that the term on the second line of \eqref{NV3} tends to zero as $j \to + \infty$. 
Consequently,
the term on the first line of \eqref{NV3} converges to zero as $j \to + \infty$
and this shows \eqref{NV1}.
The proof of \eqref{NV2} is similar. This proves \eqref{NV0}. 	
	By the definition of $L$ in \eqref{Ldef},  
	we obtain the reverse inequality, so that $\textstyle{L=\hat J_r(\bar\gamma,\bar\tau)}.$ 
\end{proof}
\begin{proposition}\cite{HNPW} \label{cfdrdc}
	The reduced cost functional $\hat{J}_r$
		is continuously Fréchet differentiable. For every $(\gamma,\tau)$ in $H^1(0,T;\R^2) \times [3r,T]$, we have
		\begin{equation*}
			D_{(\gamma,\tau)} \hat J_r(\gamma,\tau)\cdot(\delta\gamma,\delta \tau)=D_\gamma \hat J_r(\gamma,\tau)\cdot\delta\gamma+D_\tau \hat J_r(\gamma,\tau)\cdot\delta\tau
		\end{equation*}
		for every $(\delta \gamma,\delta \tau) \in H^{1}(0,T;\mathbb{R}^{2})\times \mathbb{R}_{+}([3r,T]-\tau)$ with $D_\gamma \hat J_r(\gamma,\tau)$ and $D_\tau \hat J_r(\gamma,\tau)$ given by
	\begin{equation}\label{dervingamma}
		\begin{aligned}
			D_\gamma \hat J_r(\gamma,\tau)\cdot\delta\gamma  &=\textstyle{\frac{\lambda_Q}{r}\int_{\frac{\tau}{2}-\frac{r}{2}}^{\frac{\tau}{2}+\frac{r}{2}}\int_\Omega}\big(S(\gamma)(x,t)-y_Q(x,t)\big)\cdot  DS(\gamma)\delta\gamma(x,t) \d x \d t \\&\quad +\textstyle{\frac{\lambda_\Omega}{r}\int_{\tau -r}^{\tau}\int_\Omega} \big(S(\gamma)(x,t)-y_\Omega(x,t)\big)\cdot DS(\gamma)\delta\gamma(x,t)  \d x \d t \\ &\quad+\textstyle{\int_0^T\int_\Omega}  \nabla S(\gamma)(x,t)  \cdot\nabla (DS(\gamma)\delta\gamma)(x,t)  \d x \d t
			\\&\quad
			+{\lambda_\gamma}(\gamma,\delta\gamma)_{H^1(0,T;\mathbb{R}^2)}\  ,
	\end{aligned}\end{equation}
and
	\begin{equation}\label{dervintime}
			\begin{aligned}
				D_\tau \hat J_r(\gamma,\tau)
		&=
\frac{\lambda_Q}{4 r}\Big[\int_{\Omega}|S(\gamma)-y_Q|^2(x,{\tfrac{\tau}{2}+\tfrac{r}{2}}) \d x-\int_\Omega |S(\gamma)-y_Q|^2(x,{\tfrac{\tau}{2}-\tfrac{r}{2}}) \d x\Big]
\\ &\quad
+  \frac{\lambda_\Omega}{2 r}\Big[\int_{\Omega}|S(\gamma)-y_\Omega|^2(x,\tau)  \d x  -\int_\Omega |S(\gamma)-y_\Omega|^2(x,\tau-r)  \d x\Big]+ \beta_T.
\end{aligned} \end{equation} 
\end{proposition}

\begin{definition} \cite{HNPW}
The adjoint system of our problem is the linear backward parabolic boundary value problem
	\begin{equation}\label{eqn:adjointsystem}
		\begin{aligned}
			\rho c	 p_{t}  +\kappa\Delta p&=\Delta \bar y-\frac{1}{r}\big(\lambda_Q		\mathds{1}_{[{\frac{\bar\tau}{2}-\frac{r}{2}},{\frac{\bar\tau}{2}+\frac{r}{2}}]}(t)(\bar y-y_Q) +\lambda_\Omega\mathds{1}_{[\bar{\tau}-r,\bar{\tau}]}(t) 		 (\bar y-y_\Omega)\big)&& \text{in }  Q,\\
			\kappa \partial_\nu p+hp &= \partial_\nu \bar y && \text{on } \Sigma, \\
			p(\cdot,T) &= 0 && \text{in } \Omega,
		\end{aligned}
\end{equation} 
where $\Sigma :=\Gamma\times]0,T[$.
\end{definition}
\begin{definition} \cite{HNPW} \label{defweakform}
	Let $(\bar\gamma,\bar\tau)$ be an optimal control of \eqref{ROCP} with associated state $\bar{y}=S(\bar\gamma)$. A function $p\in W(0,T)$ is said to be a weak solution to the adjoint problem \eqref{eqn:adjointsystem}, if $p(\cdot,T)=0$ in $\Omega$ and
	\begin{equation}\label{varformofAdjpb}\begin{aligned}
			\MoveEqLeft
			-\rho c\int_0^T \langle p_t(.,t),\ v(.,t)\rangle _{H^1(\Omega)^{\star},H^1(\Omega)}\d t
			+
			\kappa\iint_Q \nabla p\cdot\nabla v  \d x \d t 
			+h\iint_\Sigma p v \d S(x)\d t
			\\ &=
			\iint_Q\nabla S(\bar\gamma)\cdot\nabla v \d x \d t
			+\frac{\lambda_Q}{r}\int_{\frac{\bar\tau}{2}-\frac{r}{2}}^{\frac{\bar\tau}{2}+\frac{r}{2}}\int_\Omega (S(\bar\gamma)-y_Q) v \d x \d t
			\\&\qquad
			+\frac{\lambda_\Omega}{r}\int_{\bar\tau-r}^{\bar\tau}\int_\Omega (S(\bar\gamma)-y_\Omega)(x,t) v(x,t) \d x \d t,
	\end{aligned}\end{equation} for every $v\in L^2(0,T;H^1(\Omega)).$ \end{definition}
\begin{theorem} \cite{HNPW}
	Let $(\bar\gamma,\bar\tau)$ be an optimal control and $\bar y=S(\bar \gamma)$  be the corresponding state.
	Then the backward parabolic problem \eqref{eqn:adjointsystem}  has a unique weak solution $p\in W(0,T)$.
\end{theorem}
We start by introducing the tangent cone.
\begin{definition} \cite{HNPW} \cite[Section~2.2.4]{BoSh} \label{def:caddir}
	Let $V$ be a normed vector space and $U$ a non-empty subset of $V$. For every $u\in U$, the tangent cone at $u$ is
	\begin{equation} \label{caddir}
		\TT_U(u) := \set*{d\in V;\exists ( u_{n} ) \subset U,
			\, u_n \to u,
			\, \exists (t_n) \subset \R_+^\ast,
			\, t_n \to 0,
			% \text{ with }
			\frac{u_n - u}{t_n} \to d
		}.
	\end{equation}
\end{definition}
This cone is always closed but not necessarily convex, see \cite[p.45]{BoSh}.
\begin{proposition} \cite{HNPW} \label{NOC}
	If $(\bar \gamma,\bar\tau) \in \Uad \times  [3r,T]$ is an optimal control to the minimizing problem \eqref{ROCP} with associated state $\bar y$ and $p\in W(0,T)$ the corresponding adjoint state that solves \eqref{eqn:adjointsystem}, then $(\bar \gamma,\bar\tau)$ verifies the variational inequalities
	\begin{equation}\label{varineq}
		\begin{aligned}
			D_\gamma \hat J_r(\bar\gamma,\bar\tau)\cdot\delta\bar\gamma
		&=
	c_R\textstyle\iint_{\Sigma_1}e^{w(\bar\gamma)(x,t)}\tilde\gamma(\bar\gamma)(x,t)\cdot(\delta\bar\gamma)(t)p(x,t) \d S(x)\d t
\\&\qquad
+ {\lambda_\gamma}\textstyle\int_0^T{\bar\gamma'(t)\cdot\delta\bar\gamma'(t)}\d t
\\&
\geq 0, \quad \text{ for all } \delta\bar\gamma\in \TT_{\Uad}(\bar\gamma),
\end{aligned}
	\end{equation}
	where $c_R:=\frac{8 \alpha P}{\pi R^4}$, $\TT_{\Uad}(\bar\gamma)$ denotes the tangent cone to $\Uad$ at $\bar{\gamma}$ defined by \eqref{caddir} with $V=H^{1}(0,T;\mathbb{R}^{2})$,
	and
	\begin{equation} \label{varineq2}
		\begin{aligned}D_\tau \hat{J}_r(\bar\gamma,\bar\tau)(s-\bar\tau)
	% 		&=
	% 		\textstyle{\frac{\lambda_Q}{4}\frac{s-\bar\tau}{r}\left[\norm[\big]{S(\bar\gamma)(\cdot,{\tfrac{\bar\tau}{2}+\tfrac{r}{2}})-y_Q(\cdot,{\tfrac{\bar\tau}{2}+\tfrac{r}{2}})}^2_{L^2(\Omega)}-\norm[\big]{
	% 		S(\bar\gamma)(\cdot,{\tfrac{\bar\tau}{2}-\tfrac{r}{2}})-y_Q(\cdot,{\tfrac{\bar\tau}{2}-\tfrac{r}{2}})}^2_{L^2(\Omega)}\right]}\\ &\qquad+ \textstyle{\frac{\lambda_\Omega}{2}\frac{s-\bar\tau}{r}\left[\norm[\big]{S(\bar\gamma)(\cdot,\bar\tau)-y_\Omega(\cdot,\bar\tau)}^2_{L^2(\Omega)}-\norm[\big]{
	% S(\bar\gamma)(\cdot,\bar\tau-r)-y_\Omega(\cdot,\bar\tau-r)}^2_{L^2(\Omega)}\right]}\\&\qquad + \beta_T(s-\bar\tau)
	\geq 0 \text{ for all } s\in [3r,T],\end{aligned} \end{equation}
	where $D_\tau \hat{J}_r(\bar\gamma,\bar\tau)$ was given in \eqref{dervintime}.
\end{proposition}

Let us now recall the main second order results obtained in \cite{HNPW}. 

Firstly, we have proved that the state mapping $S$ is two times continuously Fréchet differentiable from $H^1(0,T;\mathbb{R}^2)$ into $W(0,T)$ and have given the expression of its second order Fréchet derivative (let us recall that in \cite{HNPW}, the state mapping is denoted $G$ but to avoid confusion with \eqref{defopG}, we have denoted it $S$ in this paper). More precisely, we have proved the following proposition:
\begin{proposition} \cite{HNPW}
The control to state map $S$ is twice continuously Fréchet differentiable from $H^{1}(0,T;\mathbb{R}^2)$ into $W(0,T)$. The second derivative $D^{2}S(\gamma)$ at point $\gamma \in H^{1}(0,T;\mathbb{R}^2)$ is the bilinear continuous mapping which sends $(\delta \gamma_1,\delta \gamma_2) \in [H^{1}(0,T;\mathbb{R}^2)]^{2}$ to $\frac{8\alpha P}{\pi R^4}y$, where $y=y(\delta \gamma_1,\delta \gamma_2)\in W(0,T)$ is the weak solution of
\begin{equation}\label{addivbp2}
	\begin{cases}
		\rho c \partial_t y - \kappa \Delta y = 0 \qquad  \quad \ \   &\text{ in }   Q,\\
		\kappa \partial_\nu y + hy=e^{w(\gamma)}(\frac{4}{R^2}\tilde\gamma(\gamma).\delta\gamma_1\tilde\gamma(\gamma).\delta\gamma_2-\delta\gamma_1.\delta\gamma_2) \quad   &\text{ on }   \Sigma_1, \\
		\kappa \partial_\nu y + hy=0 \quad\quad \ \ \ &  \text{ on }   \Sigma_2 \cup \Sigma_3,\\
		y(\cdot,0) = 0\qquad  \qquad  &\text{ in  }\  \Omega.
	\end{cases}
\end{equation}
\end{proposition}

Secondly, we have proved in \cite{HNPW}, that under the additional assumptions $\dot{y}_{B}\in L^{2}(\Sigma_{3})$,
$y_{0}\in L^{2}(\Omega)$, $D_{t}y_{Q}\in C([r,\frac{2}{3}T];L^{2}(\Omega))$,
$D_{t}y_{\Omega}\in C([2r,T];L^{2}(\Omega))$, that $D_tS(\gamma) \in C([\varepsilon,T];L^2(\Omega))$ for every $\epsilon >0$, that the mapping $H^1(0,T;\mathbb{R}^2) \rightarrow C([\varepsilon,T];L^2(\Omega)): \gamma \mapsto D_t(S(\gamma)|_{[\varepsilon,T]})$ is continuous and that the reduced cost
functional
\[
\hat{J}_{r}:H^{1}(0,T;\mathbb{R} ^{2})\times\lbrack3r,T]\rightarrow\mathbb{R} :\left(  \gamma,\tau\right)  \longmapsto\hat{J}_{r}\left(  \gamma,\tau\right)
\]
is two times continuously Fr\'{e}chet differentiable. We also have the following expressions for the second order Fréchet derivatives of the reduced cost functional $\hat{J}_r$  \cite{HNPW}
\begin{equation} \label{DgghatJr}
	\begin{aligned}
		\MoveEqLeft
		D_{\gamma\gamma}  \hat J_r(\gamma,\tau)\cdot[\delta\gamma_1,\delta\gamma_2]
		\\
		&=
		\frac{\lambda_Q}{r}\int_{\frac{\tau}{2}-\frac{r}{2}}^{\frac{\tau}{2}+\frac{r}{2}}\int_\Omega\Big((S(\gamma)-y_Q) D_{\gamma\gamma}S(\gamma)\cdot[\delta\gamma_1,\delta\gamma_2] +v(\gamma,\delta\gamma_2) .v(\gamma,\delta\gamma_1)  \Big) \d x \d t 
		\\
		&\quad
		+\frac{\lambda_\Omega}{r}\int_{\tau -r}^{\tau}\int_\Omega  \Big((S(\gamma)-y_\Omega) D_{\gamma\gamma}S(\gamma)\cdot[\delta\gamma_1,\delta\gamma_2] +v(\gamma,\delta\gamma_2) .v(\gamma,\delta\gamma_1)  \Big)  \d x \d t 
		\\
		&\quad
		+\int_0^T\int_\Omega\Big(  \nabla v(\gamma,\delta\gamma_2)  .\nabla v(\gamma,\delta\gamma_1) + \nabla S(\gamma) .\nabla \Big(D_{\gamma\gamma}S(\gamma).[\delta\gamma_1,\delta\gamma_2]\Big) \Big)  \d x \d t 
		\\
		&\quad
		+{\lambda_\gamma}\int_0^T \frac{\d\delta\gamma_1}{\d t}\cdot\frac{\d\delta\gamma_2}{\d t} \d t, 
\end{aligned}\end{equation}
\begin{equation} \label{DtghatJr}
\begin{aligned}
	\MoveEqLeft
	D_{\tau\gamma}  \hat J_r(\gamma,\tau).[\delta\gamma,\delta\tau]=D_{\gamma \tau}  \hat J_r(\gamma,\tau).[\delta\tau,\delta\gamma]
	\\
	&=
	\frac{\lambda_Q\delta\tau}{2r}\int_\Omega\Big((S(\gamma)-y_Q)(x,\textstyle\frac{\tau}{2}+\frac{r}{2})v(\gamma,\delta\gamma)(x,\textstyle\frac{\tau}{2}+\frac{r}{2}) 
		\\&\qquad\qquad\qquad
	-(S(\gamma)-y_Q)(x,\textstyle\frac{\tau}{2}-\frac{r}{2})v(\gamma,\delta\gamma)(x,\textstyle\frac{\tau}{2}-\frac{r}{2})\Big) \d x
	\\
	&\qquad
	+\frac{\lambda_\Omega\delta\tau}{r}\int_\Omega\Big((S(\gamma)-y_{\Omega})(x,\tau)v(\gamma,\delta\gamma)(x,\tau)
		\\&\qquad\qquad\qquad
		-
	(S(\gamma)-y_\Omega)(x,\tau-r)v(\gamma,\delta\gamma)(x,\tau-r)\Big)   \d x, 
	\end{aligned}\end{equation}
and
\begin{equation}\label{sectideriv}
\begin{aligned}
	\MoveEqLeft
	D_{\tau\tau}\hat{J}_{r}(\gamma,\tau).[\delta\tau_{1},\delta\tau_{2}%
	]
	\\
	&=\tfrac{\lambda_{Q}\delta\tau_1\delta\tau_2}{4r}\Big(\int_{\Omega}(S(\gamma)-y_{Q})(x,\textstyle\frac{\tau}{2}+\frac{r}{2})D_{t}(S(\gamma)-y_{Q})(x,\textstyle\frac{\tau}{2}+\frac{r}{2}) \d x\\
	&\qquad\qquad\qquad-\int_\Omega(S(\gamma)-y_{Q})(x,\textstyle\frac{\tau}{2}-\frac{r}%
	{2})D_{t}(S(\gamma)-y_{Q})(x,\textstyle\frac{\tau}%
	{2}-\frac{r}{2}) \d x\Big)\\
	&\quad+\tfrac{\lambda_{\Omega}\delta\tau_1\delta\tau_2}{2r}\Big(\int_{\Omega}(S(\gamma)-y_{\Omega})%
	(x,\tau)D_{t}(S(\gamma)-y_{\Omega})(x,\tau)\\
	&\qquad\qquad\qquad
	-	\int_\Omega(S(\gamma)-y_{\Omega})(x,\tau-r)D_{t}(S(\gamma)-%
	y_{\Omega})(x,\tau-r) \d x\Big).
\end{aligned}
\end{equation}	
In these formulas $v(\gamma,\delta\gamma):=DS(\gamma)\delta\gamma$.
As an application of these second-order regularity results, we have given in \cite{HNPW} in case $\Gamma_S$ is a closed convex subset of $\Gamma_1 \subset \mathbb{R}^2$, for $(\bar\gamma,\bar\tau) \in \Uad\times [3r,T]$ satisfying the first order necessary conditions \eqref{varineq}-\eqref{varineq2}, a sufficient condition for $(\bar\gamma,\bar\tau)$ to be a strict local minimizer of the reduced cost functional $\hat{J}_{r}$.

\section{\textbf{Higher regularity of the trajectories for local minimizers}}
\label{sec:regularity}

In the next theorem, we prove the $H^2$-regularity of the trajectory $\bar{\gamma}$ of a \textit{local} minimizer $(\bar{\gamma
},\bar{\tau})$ of problem  \eqref{ROCP}.
We start with a lemma.

\begin{lemma} \label{cdrte}
The mapping $\mathcal{J}:H^{1}(0,T;\mathbb{R} ^{2})\rightarrow\mathbb{R} :\gamma\longmapsto\int_{0}^{T}[g(\gamma(t))]_{+}^{2} \d t$ is continuously
differentiable and
\[
D\mathcal{J}(\gamma)\delta\gamma=2\int_{0}^{T}[g(\gamma
(t))]_{+}\nabla g(\gamma(t))\cdot\delta\gamma(t) \d t,\text{ for every }%
\delta\gamma\in H^{1}(0,T;\mathbb{R} ^{2})\text{. }%
\]

\end{lemma}

\begin{proof}
	We will prove in Proposition \ref{GfD} below that the mapping
\[
G:H^{1}(0,T;\mathbb{R} ^{2})\rightarrow H^{1}(0,T):\gamma\mapsto g\circ\gamma
\]
is continuously differentiable and that
\[
DG(\gamma).\delta\gamma=(\left(  \nabla g\right)  \circ\gamma)\cdot
\delta\gamma,\text{ for every }\delta\gamma\in H^{1}(0,T;\mathbb{R} ^{2})\text{. }%
\]
Thus by the well-known chain rule,
it suffices to prove that the mapping%
\[
f:H^{1}(0,T)\rightarrow\mathbb{R} :\varphi\longmapsto\int_{0}^{T}[\varphi(t)]_{+}^{2} \d t=\left\Vert \varphi
_{+}\right\Vert _{L^{2}(0,T)}^{2}%
\]
is continuously differentiable and that%
\[
Df(\varphi).\delta\varphi=2\int_{0}^{T}[\varphi(t)]_{+}\delta
\varphi(t) \d t,\text{ for every }\delta\varphi\in H^{1}(0,T)\text{.}%
\]
This follows from standard arguments,
see, e.g., \cite[Theorem~7]{GKT}
combined with the continuous embedding $H^1(0,T) \hookrightarrow C([0,T])$.
\end{proof}

\begin{theorem} \label{regoptraj}
Let $\left(  \overline{\gamma},\bar{\tau}\right)  $ be a local minimizer of
\eqref{ROCP}.
Then $\overline{\gamma}\in H^{2}(0,T;\mathbb{R} ^{2})$ and $\overline{\gamma}^{\prime}(0)=\overline{\gamma}^{\prime}(T)=0_{\mathbb{R} ^{2}}$.
\end{theorem}
\begin{proof}
Let us recall (section \ref{sec:prelim}), that $\Gamma_{S}$ is the compact subdomain of $\Gamma_{1}\subset\mathbb{R} ^{2}$ ($\Gamma_{1}$ is an open subset of $\partial\Omega$) defined by a
function $g:\mathbb{R} ^{2}\rightarrow\mathbb{R} :x\mapsto g(x)$ of class $C^{3}$ such that $\Gamma_{S}=\{x\in\Gamma
_{1};g(x)\leq0\}$ and $\nabla g(x)\neq0_{\mathbb{R}^{2}}$ for every $x\in\partial\Gamma_{S}$. 

For $\kappa>0$, we define the \textit{regularized} cost functional:%
\begin{equation}
\tilde{J}_{\kappa}\left(  \gamma,\tau\right)  :=\hat{J}_{r}\left(  \gamma
,\tau\right)  +\kappa\int_{0}^{T}\left[  g(\gamma(t)\right]  _{+}%
^{2}\d t+\left\Vert \gamma-\overline{\gamma}\right\Vert _{L^{2}(0,T;\mathbb{R} ^{2})}^{2}+|\tau-\bar{\tau}|^{2}, \label{add2}%
\end{equation}
where the reduced cost functional $\hat J_r$ was defined in
Definition \ref{defrcf}.

For $\rho>0$, let us consider%
\begin{equation}
C:=\bar{B}(\overline{\gamma};\rho)\times([3r,T]\cap\bar{B}(\overline{\tau
};\rho)) \label{add3}%
\end{equation}
where $\bar{B}(\overline{\gamma};\rho)$ (resp. $\bar{B}(\overline{\tau};\rho
)$) denotes the closed ball of center $\overline{\gamma}$ (resp.\ $\overline{\tau}$)
and radius $\rho$ in $H^{1}(0,T;\mathbb{R} ^{2})$ (resp.\ $\mathbb{R} $).
Clearly, $C$ is a closed convex subset of the Hilbert space $H^{1}(0,T;\mathbb{R} ^{2})\times\mathbb{R} $. We fix $\rho>0$ sufficiently small such that $(\overline{\gamma
},\bar{\tau})$ is a global minimizer of $\hat{J}_{r}$ on $(\Uad\cap\bar
{B}(\overline{\gamma};\rho))\times([3r,T]\cap\bar{B}(\overline{\tau};\rho))$.

\underline{First step}: In this first step, for $\kappa >0$ fixed, we consider the minimization problem:
\begin{equation}
\min\tilde{J}_{\kappa}\left(  \gamma,\tau\right)  \text{ for }\left(
\gamma,\tau\right)  \in C. \label{add4}%
\end{equation}
We are going to prove that $\inf\limits_{(\gamma,\tau)\in C}\tilde{J}_{\kappa}\left(  \gamma,\tau\right) $ is attained at some point $\left(  \tilde{\gamma}_{\kappa},\tilde{\tau}_{\kappa
}\right)$ in $C$. 

The closed convex set $C$ is also bounded in the
Hilbert space $H^{1}(0,T;\mathbb{R} ^{2})\times\mathbb{R} $ so that it is a weakly sequentially compact subset of $H^{1}(0,T;\mathbb{R} ^{2})\times\mathbb{R} $ \cite[pp.46-47]{FrTr}.
From the proof of Theorem \ref{EGMin},
we get that $\hat J_r$ is weakly sequentially lower semicontinuous.
Owing to the compact embedding from
$H^1(0,T; \R^2)$ in $C([0,T]; \R^2)$,
one can check that the mapping $\mathcal{J}$
from \ref{cdrte}
is weakly sequentially lower semicontinuous.
Now, standard arguments apply
and yield the existence of a solution $(\tilde \gamma_\kappa, \tilde \tau_\kappa)$ of \eqref{add4}.

\underline{Second step:} In this second step, we study the behaviour of $(\tilde{\gamma}_{\kappa},\tilde{\tau}_{\kappa})$ as $\kappa$ tends to $+\infty$. In particular, we are going to prove that the sequence $\left(  \tilde{\gamma}_{\kappa}%
,\tilde{\tau}_{\kappa}\right)  _{\kappa=1}^{+\infty}$ converges strongly in $H^{1}(0,T;
\mathbb{R}^{2})\times\mathbb{R}$ to $\left(  \bar{\gamma},\bar{\tau}\right).$  

The sequence $\left(  \tilde{\gamma}_{\kappa}%
,\tilde{\tau}_{\kappa}\right)  _{\kappa=1}^{+\infty}$ is trivially bounded in
$H^{1}(0,T;\mathbb{R} ^{2})\times\mathbb{R} $ as $(\tilde{\gamma}_{\kappa},\tilde{\tau}_{\kappa})\in\bar{B}(\overline
{\gamma};\rho)\times\lbrack3r,T]$, so that it possesses a subsequence $\left(
\tilde{\gamma}_{\kappa_{l}},\tilde{\tau}_{\kappa_{l}}\right)_{l=0}^{+\infty}$, weakly convergent to some $\left(  \tilde{\gamma
},\tilde{\tau}\right)  \in H^{1}(0,T;\mathbb{R} ^{2})\times\lbrack3r,T]$. We also have that $\left(  \tilde{\gamma}%
,\tilde{\tau}\right)  \in C$, since $C$ is weakly closed.

The tuple $\left(  \bar{\gamma},\bar{\tau}\right)  $ belongs also to $C$, thus we have
\begin{equation}
\tilde{J}_{\kappa}\left(  \tilde{\gamma}_{\kappa},\tilde{\tau}%
_{\kappa}\right)  \leq\tilde{J}_{\kappa}\left(  \bar{\gamma},\bar
{\tau}\right)  =\hat{J}_{r}\left(  \bar{\gamma},\bar{\tau}\right)  \label{add7}%
\end{equation}
due to the definition \eqref{add2} of the regularized cost
functional $\tilde{J}_{\kappa}$ and due to the fact that the range of
$\bar{\gamma}$ being contained in $\Gamma_{S}$, we have $g(\bar{\gamma}(t))\leq0$
for all $t\in\lbrack0,T]$ so that $\left[  g(\bar{\gamma}(t)\right]  _{+}=0$
for all $t\in\lbrack0,T]$. Inequality \eqref{add7} implies that%
\begin{equation}
\kappa\int_{0}^{T}\left[  g(\tilde{\gamma}_{\kappa}(t))\right]
_{+}^{2}\d t\leq\hat{J}_{r}\left(  \bar{\gamma},\bar{\tau}\right)
<+\infty\label{add8}%
\end{equation}
so that $\int_{0}^{T}\left[  g(\tilde{\gamma}_{\kappa}(t))\right]  _{+}%
^{2}\d t\rightarrow0$ as $\kappa\rightarrow+\infty$.
Due to the strong convergence of
$\tilde{\gamma}_{\kappa_{l}}$ to $\tilde{\gamma}$ in $C([0,T];\mathbb{R} ^{2})$, it follows
that $\int_{0}^{T}\left[
g(\tilde{\gamma}(t)\right]  _{+}^{2}\d t=0$. Thus $g(\tilde{\gamma}(t))\leq0$
for all $t\in\lbrack0,T]$ as
$g\circ\tilde{\gamma}$ is continuous. Consequently, $\tilde{\gamma}%
(t)\in\Gamma_{S}$ for all $t\in\lbrack0,T]$. Thus $\tilde{\gamma}\in \Uad$.

By \eqref{add7}, we have a fortiori%
\begin{equation}
\hat{J}_{r}\left(  \tilde{\gamma}_{\kappa},\tilde{\tau}_{\kappa%
}\right)  +\left\Vert \tilde{\gamma}_{\kappa}-\overline{\gamma}\right\Vert
_{L^{2}(0,T;\mathbb{R} ^{2})}^{2}+|\tilde{\tau}_{\kappa}-\bar{\tau}|^{2}\leq\tilde{J}_{\kappa}\left(  \bar{\gamma},\bar{\tau}\right)  =\hat{J}_{r}\left(  \bar{\gamma
},\bar{\tau}\right)  .\label{add9}%
\end{equation}
Taking $\kappa=\kappa_{l}$ in \eqref{add9} and passing to the limit as $l \to +\infty$, we obtain by using
the weak sequential lower semicontinuity of $\hat J_r$
that
\begin{equation}
\hat{J}_{r}\left(  \tilde{\gamma},\tilde{\tau}\right)  +\left\Vert \tilde{\gamma}-\overline{\gamma}\right\Vert _{L^{2}(0,T;\mathbb{R} ^{2})}^{2}+|\tilde{\tau}-\bar{\tau}|^{2}
\leq\hat{J}_{r}\left(  \bar{\gamma
},\bar{\tau}\right)  .\label{add10}%
\end{equation}
As $\left(  \overline{\gamma},\bar{\tau}\right)  $ is a solution of the
problem $\hat{J}_{r}\left(  \overline{\gamma},\bar{\tau}\right)
=\min\limits_{\left(  \gamma,\tau\right)  \in(\Uad\times\mathbb{R} )\cap C}\hat{J}_{r}\left(  \gamma,\tau\right)  $ and as we have seen that
$\tilde{\gamma}\in \Uad$, we also have that:%
\begin{equation}
\hat{J}_{r}\left(  \bar{\gamma},\bar{\tau}\right)  \leq\hat{J}_{r}\left(
\tilde{\gamma},\tilde{\tau}\right)  .\label{add11}%
\end{equation}
From \eqref{add10} and \eqref{add11} follows that
$\tilde{\gamma
}=\overline{\gamma}$ and $\tilde{\tau}=\bar{\tau}$. Thus the subsequence
$\left(  \tilde{\gamma}_{\kappa_{l}},\tilde{\tau}_{\kappa_{l}}\right)  $
converges weakly in $H^{1}(0,T;\mathbb{R} ^{2})\times\lbrack3r,T]$ to $\left(  \bar{\gamma},\bar{\tau}\right).$ 
Further, \eqref{add9} and \eqref{add11} imply
\begin{equation*}
	\hat{J}_{r}\left(  \bar{\gamma},\bar{\tau}\right)
	=
	\lim_{l \to \infty}
	\hat{J}_{r}\left( \gamma_{\kappa_l}, \tau_{\kappa_l}\right)
	.
\end{equation*}
Since the reduced objective contains the $H^1$-seminorm of $\gamma$
and since all other terms are weakly continuous (see Proposition \ref{GWSC}),
this implies
\begin{equation*}
    \abs{ \bar{\gamma} }_{H^1(0,T;\R^2)}
    =
    \lim_{l \to \infty}
    \abs{ \gamma_{\kappa_l} }_{H^1(0,T;\R^2)}
    .
\end{equation*}
Further, $H^1(0,T;\R^2)$ is compactly embedded into $L^2(0,T;\R^2)$
and
this implies the convergence of the $L^2$-norms.
This yields
\begin{equation*}
    \norm{ \bar{\gamma} }_{H^1(0,T;\R^2)}
    =
    \lim_{l \to \infty}
    \norm{ \gamma_{\kappa_l} }_{H^1(0,T;\R^2)}
    .
\end{equation*}

Consequently,
the convergence $\gamma_{\kappa_l} \to \bar\gamma$
is strong in $H^1(0,T;\R^2)$.

The usual subsequence-subsequence argument yields the convergence of the entire sequence.

\underline{Third step}:
We evaluate the optimality conditions at the minimizer
$(\tilde \gamma_\kappa, \tilde \tau_\kappa)$.
Note that the strong convergence in $H^1(0,T;\R^2)$ implies that $\tilde{\gamma}_{\kappa}$ will
belong to the open ball $B(\overline{\gamma};\rho)$ in $H^{1}(0,T;\mathbb{R} ^{2})$ for large $\kappa$.
Consequently,
$\tilde{\gamma}_{\kappa}$ verifies the optimality
condition%
\begin{equation}
D_{\gamma}\tilde{J}_{\kappa}\left(  \tilde{\gamma}_{\kappa},\tilde{\tau
}_{\kappa}\right)  =0_{H^{1}(0,T;\mathbb{R} ^{2})^{\ast}}.\label{add13}%
\end{equation}
Thus, by \eqref{add2}, \eqref{varineq}, and Lemma \ref{cdrte}
we get
\begin{equation}%
	\begin{aligned}
		\MoveEqLeft
D_{\gamma}\tilde{J}_{\kappa}\left(  \tilde{\gamma}_{\kappa},\tilde{\tau
}_{\kappa}\right)  \delta\tilde{\gamma}
\\
&=D_{\gamma}\hat{J}_{r}(\tilde{\gamma
}_{\kappa},\tilde{\tau}_{\kappa})\delta\tilde{\gamma}
\\&\qquad
+2(\tilde{\gamma}%
_{\kappa}-\overline{\gamma},\delta\tilde{\gamma})_{L^{2}(0,T;\mathbb{R} ^{2})}+2\kappa\int_{0}^{T}\left[  g(\tilde{\gamma}_{\kappa}(t))\right]
_{+}\nabla g(\tilde{\gamma}_{\kappa}(t))\cdot\delta\tilde{\gamma}(t) \d t\\
&=c_{R}\int_{0}^{T}\int_{\Gamma_{1}}e^{w(\tilde{\gamma}_{\kappa
})(x,t)}(x-\tilde{\gamma}_{\kappa}(t))\cdot\delta\tilde{\gamma}(t)p_{\kappa
}(x,t) \d S(x)\d t+\lambda_{\gamma}\int_{0}^{T}\tilde{\gamma}_{\kappa}^{\prime
}\cdot\delta\tilde{\gamma}^{\prime}\d t\\
&\qquad
+2(\tilde{\gamma}_{\kappa}-\overline{\gamma}, \delta\tilde{\gamma})_{L^{2}(0,T;\mathbb{R} ^{2})}+2\kappa\int_{0}^{T}\left[  g(\tilde{\gamma}_{\kappa}(t))\right]
_{+}\nabla g(\tilde{\gamma}_{\kappa}(t))\cdot\delta\tilde{\gamma}(t) \d t
\\&=0,
	\end{aligned}
\label{add17}%
\end{equation}
for every $\delta\tilde{\gamma}\in H^{1}(0,T;\mathbb{R} ^{2})$, the adjoint state $p_{\kappa}:(x,t)\mapsto p_{\kappa}(x,t)$ being defined by \eqref{varformofAdjpb}, but with $S(\bar{\gamma})$ replaced by
$S(\tilde{\gamma}_{\kappa})$ and $\bar{\tau}$ by $\tilde{\tau}_{\kappa}$ in the right-hand side of equation \eqref{varformofAdjpb}.

With the abbreviation
\begin{equation}
P_{\kappa}(t):=\int_{\Gamma_{1}}c_{R}\exp(w(\tilde{\gamma}_{\kappa
})(x,t))(x-\tilde{\gamma}_{\kappa}(t))p_{\kappa}(x,t) \d S(x), \label{add20}%
\end{equation}
the equation \eqref{add17} becomes
\begin{equation}%
	\begin{aligned}
		0
		&=
		\int_{0}^{T}P_{\kappa}(t)\cdot\delta\tilde{\gamma}(t)\d t+\lambda_{\gamma}%
		\int_{0}^{T}\tilde{\gamma}_{\kappa}^{\prime}(t)\cdot\delta\tilde{\gamma
		}^{\prime}(t)\d t+2\int_{0}^{T}(\tilde{\gamma}_{\kappa}(t)-\bar{\gamma}%
		(t))\cdot\delta\tilde{\gamma}(t)\d t\\
		&\qquad
		+2\kappa\int_{0}^{T}\left[  g(\tilde{\gamma}_{\kappa}(t))\right]  _{+}\nabla
		g(\tilde{\gamma}_{\kappa}(t))\cdot\delta\tilde{\gamma}(t) \d t,
	\end{aligned}
	\label{add21}%
\end{equation}
for every $\delta\tilde{\gamma}\in H^{1}(0,T;\mathbb{R} ^{2})$.
Let us prove that $P_{\kappa}\in L^{2}(0,T;\mathbb{R} ^{2})$.
We have $p_{\kappa}\in W(0,T)\subset L^{2}(0,T;H^{1}(\Omega))$ so that
$p_{\kappa}|_{\Gamma_{1}}\in L^{2}(0,T;L^{2}(\Gamma_{1}))$.
Further, $\tilde{\gamma
}_{\kappa}\in H^{1}(0,T;\mathbb{R} ^{2})\hookrightarrow C([0,T];\mathbb{R} ^{2})$.
Finally $\exp(w(\tilde{\gamma}_{\kappa})(x,t))=\exp(-\frac{2}{R^{2}}%
|x-\tilde{\gamma}_{\kappa}(t)|^{2})\leq1$. Putting these together in equation 
\eqref{add20}, we obtain that $P_{\kappa}\in L^{2}(0,T;\mathbb{R} ^{2})$. Moreover the norm of $P_{\kappa}$ in $L^{2}(0,T;\mathbb{R} ^{2})$ is uniformly bounded as the norm of $p_{\kappa}$ is uniformly bounded
in $W(0,T)$ due to $\left\Vert \tilde{\gamma}_{\kappa}\right\Vert _{H^{1}(0,T;\mathbb{R} ^{2})}\leq C(\bar{\gamma},\rho)$ (a positive constant depending only on $\bar{\gamma}$ and $\rho$), Proposition \ref{GWSC}, \eqref{varformofAdjpb}, and  \cite[Ch.XVIII]{9} or \cite[Ch.11]{Chi}.
We also know that $\tilde{\gamma}_{\kappa}$ and $\bar{\gamma}$
belong to $H^{1}(0,T;\mathbb{R} ^{2})$ so that $\tilde{\gamma}_{\kappa}-\bar{\gamma}\in H^{1}(0,T;\mathbb{R} ^{2})$. As $\tilde{\gamma}_{\kappa} \in \bar{B}(\bar\gamma,\rho)$, $\left[  g(\tilde{\gamma}_{\kappa%
}(.)\right]  _{+}\in L^{\infty}(\Gamma_{S,\varepsilon})$ and  $\nabla g(\tilde{\gamma}_{\kappa}(\cdot))\in L^{\infty}%
(\Gamma_{S,\varepsilon})^{2}$ with uniform bounds on their $L^{\infty}$-norms.
Thus, \eqref{add21} is the weak formulation of a second-order ODE in $\R^2$
and standard arguments yield
\begin{equation}
\lambda_{\gamma}\tilde{\gamma}_{\kappa}^{{\prime\prime
}}(t)
=
P_{\kappa}(t)
+2(\tilde{\gamma}_{\kappa}(t)-\bar{\gamma}(t))+2\kappa\left[
g(\tilde{\gamma}_{\kappa}(t))\right]  _{+}\nabla g(\tilde{\gamma}%
_{\kappa}(t))
\label{add22}%
\end{equation}
and the boundary conditions
\begin{equation}
\tilde{\gamma}_{\kappa}^{\prime}(0)=0_{\mathbb{R} ^{2}}\text{ and }\tilde{\gamma}_{\kappa}^{\prime}(T)=0_{\mathbb{R} ^{2}}. \label{add24}%
\end{equation}

\underline{Fourth step}:
We derive a uniform bound on the $H^2$-norm of $\gamma_\kappa$
and pass to the limit.
We test equation \eqref{add22} with
$\tilde{\gamma}_{\kappa}^{{\prime\prime}}(t)$ and obtain
\begin{equation}%
	\begin{aligned}
\lambda_{\gamma}\int_{0}^{T}|\tilde{\gamma}_{\kappa}^{{\prime\prime}%
}(t)|^{2}\d t&=\int_{0}^{T}P_{\kappa}(t)\cdot\tilde{\gamma}_{\kappa%
}^{\prime\prime}(t)\d t+2\int_{0}^{T}(\tilde{\gamma}_{\kappa}(t)-\bar{\gamma}(t))\cdot\tilde{\gamma}_{\kappa}^{{\prime\prime}}(t)\d t\\
&\qquad +2\kappa\int_{0}^{T}\left[  g(\tilde{\gamma}_{\kappa}(t)\right]
_{+}\nabla g(\tilde{\gamma}_{\kappa}(t))\cdot\tilde{\gamma}_{\kappa%
}^{{\prime\prime}}(t)\d t.
	\end{aligned}
\label{add30}%
\end{equation}
Let us work on the third term on the right-hand side of equation \eqref{add30}.
By taking into account \eqref{add24} and by using integration by parts \cite[Theo. 2.8 p.28]{Chi},
we obtain
\begin{equation}%
	\begin{aligned}
		\MoveEqLeft
2\kappa\int_{0}^{T}\left[  g(\tilde{\gamma}_{\kappa}(t)\right]
_{+}\nabla g(\tilde{\gamma}_{\kappa}(t))\cdot\tilde{\gamma}_{\kappa%
}^{{\prime\prime}}(t)\d t
\\
&=-2\kappa\int_{0}^{T}\mathbf{1}_{\{g(\tilde{\gamma}_{\kappa}%
(t))>0\}}(t)\left[  \nabla g(\tilde{\gamma}_{\kappa}(t))\cdot\tilde
{\gamma}_{\kappa}^{{\prime}}(t)\right]  ^{2}\d t\\
&\qquad
-2\kappa\int_{0}^{T}\left[  g(\tilde{\gamma}_{\kappa}(t)\right]
_{+}
{\tilde\gamma}_{\kappa}^{{\prime}}(t)^\top
\nabla^{2}g\left(  \tilde{\gamma }_{\kappa}(t)\right)
{\tilde\gamma}_{\kappa}^{{\prime}}(t)
\d t
\\
&\leq
-2\kappa\int_{0}^{T}\left[  g(\tilde{\gamma}_{\kappa}(t)\right]
_{+}
{\tilde\gamma}_{\kappa}^{{\prime}}(t)^\top
\nabla^{2}g\left(  \tilde{\gamma }_{\kappa}(t)\right)
{\tilde\gamma}_{\kappa}^{{\prime}}(t)
\d t
.
	\end{aligned}
\label{add31}%
\end{equation}
To bound $\kappa\left[  g(\tilde{\gamma}_{\kappa}(t)\right]  _{+}$,
let us multiply both sides of the ODE in $\mathbb{R}^2$  
\eqref{add22} by $\frac{\nabla g(\tilde{\gamma}_{\kappa}(t))}{|\nabla
g(\tilde{\gamma}_{\kappa}(t))|^{2}}$, so that we obtain:%
\begin{equation}%
\lbrack P_{\kappa}(t)+2(\tilde{\gamma}_{\kappa}(t)-\bar{\gamma
}(t))-\lambda_{\gamma}\tilde{\gamma}%
_{\kappa}^{{\prime\prime}}(t)]\cdot\frac{\nabla g(\tilde{\gamma}%
_{\kappa}(t))}{|\nabla g(\tilde{\gamma}_{\kappa}(t))|^{2}}
+2\kappa\left[  g(\tilde{\gamma}_{\kappa}(t))\right]  _{+}=0.
\label{add33}%
\end{equation}
On the set $\{t\in [0,T];\left[  g(\tilde{\gamma}_{\kappa%
}(t))\right]  _{+}\neq0\}=\{t\in [0,T];g(\tilde{\gamma}_{\kappa%
}(t))>0\}\subset\{t\in [0,T];\tilde{\gamma}_{\kappa}(t)\in
\complement\Gamma_{S}\}$, we have that $|\nabla g(\tilde{\gamma}_{\kappa%
}(t))|$ is uniformly positive for large $\kappa$ as $\tilde{\gamma}_{\kappa%
}\rightarrow\bar{\gamma}$ in $C([0,T];\mathbb{R} ^{2})$, $R(\bar{\gamma})\subset\Gamma_{S}$ and $\inf\limits_{x\in
\partial\Gamma_{S}}|\nabla g(x)|>0$. So by \eqref{add33}, we
have%
\begin{equation}
\left\Vert 2\kappa\left[  g(\tilde{\gamma}_{\kappa}(.))\right]
_{+}\right\Vert _{L^{2}(0,T)}\leq C(1+\left\Vert \tilde{\gamma}_{\kappa%
}^{{\prime\prime}}\right\Vert _{L^{2}(0,T;\mathbb{R} ^{2})}), \label{add35}%
\end{equation}
for some fixed positive constant $C$.

Using once again the convergence
$\tilde{\gamma}_{\kappa}\rightarrow\bar{\gamma}$ in
$C([0,T];\mathbb{R} ^{2})$ when $\kappa \to + \infty$
we get the bound
\begin{equation}
	\abs{
{\tilde\gamma}_{\kappa}^{{\prime}}(t)^\top
\nabla^{2}g\left(  \tilde{\gamma }_{\kappa}(t)\right)
{\tilde\gamma}_{\kappa}^{{\prime}}(t)
}
\leq C|\tilde{\gamma
}_{\kappa}^{\prime}(t)|^{2} \label{add36}%
\end{equation}
for some fixed constant $C$. 
By combining
\eqref{add30}, \eqref{add31}, \eqref{add35}, and \eqref{add36},
we obtain
\begin{equation}%
\left\Vert \tilde{\gamma}_{\kappa}^{{\prime\prime}%
}\right\Vert _{L^{2}(0,T;\mathbb{R} ^{2})}^{2}
\lesssim
\left\Vert \tilde{\gamma}_{\kappa}^{{\prime\prime}}\right\Vert _{L^{2}(0,T;\mathbb{R} ^{2})}
+
( 1 +
\left\Vert \tilde{\gamma}_{\kappa}^{{\prime\prime}}\right\Vert _{L^{2}(0,T;\mathbb{R} ^{2})}
)
\left\Vert \tilde{\gamma}_{\kappa}^{{\prime}}\right\Vert _{L^{2}(0,T;\mathbb{R} ^{2})}
\left\Vert \tilde{\gamma}_{\kappa}^{{\prime}}\right\Vert _{L^{\infty}(0,T;\mathbb{R} ^{2})}
.
\label{add41}%
\end{equation}
Note that also the $L^2$-norm of $\tilde\gamma_\kappa'$
is uniformly bounded and it remains to derive an estimate
for the $L^\infty$-norm.

We know that $\tilde{\gamma}_{\kappa}\in H^{2}(0,T;\mathbb{R} ^{2})$ and that $\tilde{\gamma}_{\kappa}^{\prime}(0)=\tilde{\gamma}_{\kappa
}^{\prime}(T)=0_{\mathbb{R} ^{2}}$. Consequently, we have for every $t\in\lbrack0,T]$%
\[
\frac{1}{2}\left\vert \tilde{\gamma}_{\kappa}^{\prime}(t)\right\vert _{\mathbb{R} ^{2}}^{2}=\int_{0}^{t}\tilde{\gamma}_{\kappa}^{\prime}(s)\cdot\tilde{\gamma
}_{\kappa}^{{\prime\prime}}(s)\d s
\quad\text{and}\quad
\frac{1}{2}\left\vert \tilde{\gamma}_{\kappa}^{\prime}(t)\right\vert _{\mathbb{R} ^{2}}^{2}=-\int_{t}^{T}\tilde{\gamma}_{\kappa}^{\prime}(s)\cdot\tilde{\gamma
}_{\kappa}^{{\prime\prime}}(s)\d s.
\]
Summing up these two equalities, we obtain%
\[
\left\vert \tilde{\gamma}_{\kappa}^{\prime}(t)\right\vert _{\mathbb{R} ^{2}}^{2}\leq\int_{0}^{T}|\tilde{\gamma}_{\kappa}^{\prime}(s)\cdot
\tilde{\gamma}_{\kappa}^{{\prime\prime}}(s)|\d s\leq\left\Vert \tilde{\gamma
}_{\kappa}^{\prime}\right\Vert _{L^{2}(0,T;\mathbb{R} ^{2})}\left\Vert \tilde{\gamma}_{\kappa}^{{\prime\prime}}\right\Vert
_{L^{2}(0,T;\mathbb{R} ^{2})},
\]
so that
$\left\Vert \tilde{\gamma}_{\kappa}^{\prime
}\right\Vert _{L^{\infty}(0,T;\mathbb{R} ^{2})}\leq\left\Vert \tilde{\gamma}_{\kappa}^{\prime}\right\Vert _{L^{2}(0,T;\mathbb{R} ^{2})}^{1/2}\left\Vert \tilde{\gamma}_{\kappa}^{{\prime\prime}}\right\Vert
_{L^{2}(0,T;\mathbb{R} ^{2})}^{1/2}
\lesssim
\left\Vert \tilde{\gamma}_{\kappa}^{{\prime\prime}}\right\Vert
_{L^{2}(0,T;\mathbb{R} ^{2})}^{1/2}
$.
We insert this estimate in \eqref{add41}
and obtain
\begin{equation}
\left\Vert \tilde{\gamma}_{\kappa}^{{\prime\prime}%
}\right\Vert _{L^{2}(0,T;\mathbb{R} ^{2})}^{2}
\lesssim
\left\Vert \tilde{\gamma}_{\kappa}^{{\prime\prime}%
}\right\Vert _{L^{2}(0,T;\mathbb{R} ^{2})}
+
\left\Vert \tilde{\gamma}_{\kappa}^{{\prime\prime}%
}\right\Vert _{L^{2}(0,T;\mathbb{R} ^{2})}^{3/2}
 +
\left\Vert \tilde{\gamma}_{\kappa}^{{\prime\prime}%
}\right\Vert _{L^{2}(0,T;\mathbb{R} ^{2})}^{1/2}\ .
\label{add42}%
\end{equation}
This implies that the norms $\left\Vert \tilde{\gamma}_{\kappa%
}^{{\prime\prime}}\right\Vert _{L^{2}(0,T;\mathbb{R} ^{2})}$ are uniformly bounded.
Consequently,
the sequence $\left(  \tilde{\gamma}%
_{\kappa}^{{\prime\prime}}\right)_{\kappa \geq1}$
converges weakly towards $\bar\gamma$
also in the space
$H^2(0,T;\R^2)$.
Passing to the limit with the boundary conditions
\eqref{add24}, we also have $\overline{\gamma
}^{\prime}(0)=\overline{\gamma}^{\prime}(T)=0_{\mathbb{R} ^{2}}$.
\end{proof}

\section{\textbf{General constrained optimization problems in Banach spaces}}
\label{subsec:soc}
In order to obtain second-order optimality conditions, sufficient or necessary, without gap between, when $\Gamma_S$ is not necessarily convex but possesses a smooth boundary,
we will reformulate our minimization problem \eqref{ROCP} in the general setting studied in
\cite[Section~5.3]{GW} in order to apply \cite[Theorem~5.7]{GW}.
The constrained optimization problem studied in \cite{GW} is
\begin{equation} \label{genminpb}
	\text{minimize}  \  f(x) \text{ subject to }g(x)\in K.
\end{equation}
Here, $X$ and $Y$ are Banach spaces, $f:X \rightarrow \mathbb{R}$, $g:X \rightarrow Y$ are twice continuously Fréchet-differentiable functions and the set $K$ is a non-empty, closed and convex subset of the Banach space $Y$.
The feasible set of \eqref{genminpb}
is defined by
\begin{equation} \label{deffeas}
\mathcal{F}:=\{x\in X;g(x)\in K\}.
\end{equation}

\subsection{\textbf{Cones and polyhedricity}}
We define some cones that will be needed later on.
Let $K$ be a closed and convex subset of some Banach space $X$ and let $\bar{x}\in K$ be given.
We define the radial cone and the normal cone of $K$ at $\bar x$
via
\begin{align*}
	\mathcal{R}_K(\bar{x}) &:= \{r(x-\bar{x}); \ x\in K ,\ r\in\R^+\}
	,\\
	\mathcal{N}_K(\bar{x}) &:= \{x^\star\in X^\star; \langle x^\star,x-\bar{x}\rangle_{X^\star,X}\leq 0, \forall x \in K \}
	,
\end{align*}
respectively.
The tangent cone from Definition~\ref{def:caddir} satisfies
\begin{equation*}
	\mathcal{T}_K(\bar x) = \cl(\mathcal{R}_K(\bar{x}))
	\qquad\text{and}\qquad
	\mathcal{N}_K(\bar{x}) = \mathcal{T}_K(\bar x)^\circ
	,
\end{equation*}
where $\cl(\cdot)$ denotes the closure
and $(\cdot)^\circ$ denotes the polar cone,
i.e.,
\begin{equation*}
	A^\circ
	:=
	\{
		x^\star \in X^\star
		;
		\forall x \in A : \langle x^\star, x \rangle_{X^\star, X} \le 0
	\}
\end{equation*}
for $A \subset X$.

We will need the following result for Cartesian products.
\begin{proposition} \label{rcpr}
	Let $Y,Z$ be Banach spaces, $K \subset Y, L \subset Z$ be closed and convex sets.
	For all $(\xi, \eta) \in K \times L$ we have
	\begin{align*}
		\mathcal{R}_{K\times L}(\xi,\eta) &= \mathcal{R}_{K}(\xi)\times\mathcal{R}_{L}(\eta)
		, \\
		\mathcal{T}_{K\times L}(\xi,\eta) &= \mathcal{T}_{K}(\xi)\times\mathcal{T}_{L}(\eta)
		, \\
		\mathcal{N}_{K\times L}(\xi,\eta) &= \mathcal{N}_{K}(\xi)\times\mathcal{N}_{L}(\eta)
		.
	\end{align*}
\end{proposition}
The first identity follows from a simple calculation owing to the
convexity of $K$ and $L$.
The second identity follows by taking the closure
and taking polars yield the third one.

In order to obtain sharp second-order conditions,
we utilize the notion of polyhedricity.
\begin{definition} \cite{GW}
Let $K$ be a closed and convex subset of the Banach space $Y$. 
We say that $K$ is polyhedric at $(y,\mu)$ with $y\in K$ and $\mu\in\mathcal{N}_{K}(y)$ if
 \begin{equation*}
\mathcal{T}_{K}(y)\cap\mu^{\perp}=\cl(\mathcal{R}_{K}(y)\cap\mu^{\perp}),
 \end{equation*}
where
\begin{equation*}
	\mu^{\perp}
	=
	\{z \in Y; \dual{\mu}{z}_{Y^*,Y} = 0\}
\end{equation*}
denotes the kernel of the continuous linear form $\mu$ on $Y$. 
We say that $K$ is polyhedric at $y\in K$, if $K$ is polyhedric at $(y, \mu)$ for all $\mu \in \mathcal{N}_{K}(y)$.
Finally, we say that $K$ is polyhedric, if it is polyhedric at all $y \in K$.  
\end{definition}
 
\begin{example} \label{exa:poly}
 \begin{enumerate}
 \item
 Let us consider $Y=H^1(0,T)$ where $T>0$. Then 
 \begin{equation*}
 K=H^1(0,T)_{-}:=\{\phi \in H^1(0,T); \phi \leq 0 \}
 \end{equation*}
 is polyhedric. This follows from \cite[Example 4.21.(3)]{GW} or by applying Theorem 3.58 of \cite{BoSh}.
 \item
	 Let $K \subset \mathbb{R}^n$ be a convex polyhedron,
	 i.e., the intersection of finitely many closed halfspaces.
	 Then $K$ is polyhedric,
	 since $\RR_K(\bar x)$ is closed for all $\bar x \in K$.
 \end{enumerate}
\end{example}

By utilizing Proposition~\ref{rcpr},
we can show that Cartesian products of polyhedric sets are again polyhedric:
\begin{proposition} \label{prop:pps}
	Let $Y,Z$ be Banach spaces, $K \subset Y, L \subset Z$ be closed and convex sets.
	If $K$ and $L$ are polyhedric
	then $K \times L$ is also polyhedric.
\end{proposition} 
\begin{proof}
Let $(y,z) \in K \times L$
and $(\mu,\lambda) \in \mathcal{N}_{K \times L}(y,z)$ be given.
We have
\begin{align*}
	\MoveEqLeft
\mathcal{T}_{K \times L}(y,z)\cap (\mu,\lambda)^{\perp}
\\
&=
(\mathcal{T}_{K}(y)\cap \mu ^{\perp}) \times (\mathcal{T}_{L}(z)\cap \lambda ^{\perp})
=
\cl(\mathcal{R}_{K}(y)\cap \mu ^{\perp}) \times \cl(\mathcal{R}_{L}(z)\cap \lambda ^{\perp})
\\
&=
\cl(\mathcal{R}_{K}(y)\cap \mu ^{\perp}) \times \mathcal{R}_{L}(z)\cap \lambda ^{\perp})
=
\cl(\mathcal{R}_{K \times L}(y,z)\cap (\mu,\lambda)^{\perp}).
\end{align*}
Thus, $K \times L$ is polyhedric.
\end{proof}  

\subsection{\textbf{First-order optimality conditions and constraint qualifications}}
We briefly recall first-order optimality conditions for Problem \eqref{genminpb} and
constraint qualifications which imply the existence (and uniqueness) of Lagrange multipliers.

The Lagrangian associated to Problem \eqref{genminpb} is
\begin{equation}
L:X \times Y^\star \rightarrow \mathbb{R}:(x,\lambda)\mapsto L(x,\lambda):=f(x)+\langle \lambda,g(x)\rangle _{Y^\star,Y}.
\end{equation}

\begin{definition}\label{statpoint}
A functional $\bar{\lambda} \in Y^{\star}$ is said to be a Lagrange multiplier at $\bar{x} \in \mathcal{F}$ if and only if
\begin{align*}
	\bar{\lambda} &\in \mathcal{N}_K(g(\bar{x})) ,
	\\
	L'(\bar{x},\bar{\lambda}) &= f'(\bar{x})+g'(\bar{x})^\star \bar{\lambda}=0_{X^\star} .
\end{align*}
Here, $L'$ denotes the partial derivative of $L$ with respect to $x$.

The ordered pair $(\bar{x},\bar{\lambda})$ is called a KKT point of \eqref{genminpb} and $\bar{x}$ is said to be a stationary point.
\end{definition}

From \cite[Lemma~3.7(i)]{BoSh}
we get the following result.
\begin{proposition} \label{implfoc}
If $\bar{x} \in \mathcal{F}$ is a local minimizer of Problem \eqref{genminpb}, then $\left\langle f^{\prime}(\bar x)  ,d\right\rangle \geq0$ for every $d \in \TT_{\mathcal{F}}(\bar x)$,
where $\TT_{\mathcal{F}}(\bar x)$ denotes the tangent cone to $\mathcal{F}$ at $\bar x$ defined by \eqref{caddir} with $V=X$.
\end{proposition}

In what follows,
we will assume that
\begin{equation} \label{gcqc}
Y=g'(\bar{x})X-\mathcal{R}_{K}(g(\bar{x}))
\end{equation}
holds,
which is the so-called
Robinson-Zowe-Kurcyusz Constraint Qualification at $\bar{x}$.
Note that this implies
\begin{equation} \label{acq}
	\TT_{\mathcal{F}}(\bar x)
	=
	\{
		d \in X;
		g'(\bar x) d \in \TT_K(g(\bar x))
	\}
	,
\end{equation}
see \cite{ZK}.

\begin{theorem} \cite[Theorem 3.1]{ZK} \label{exLagMult}
Let $\bar{x}\in \mathcal{F}$ be given
and assume that the Robinson-Zowe-Kurcyusz Constraint Qualification \eqref{gcqc} is satisfied at $\bar{x}$.
Further,
suppose that
$\langle f'(\bar x),d\rangle _{X^{\star},X}\geq 0$ for all $d \in \TT_{\mathcal{F}}(\bar x)$.
Then the set of Lagrange multipliers at $\bar{x}$ is non-empty.
In particular, if $\bar{x}\in \mathcal{F}$ is a local minimizer of Problem \eqref{genminpb} satisfying the Robinson-Zowe-Kurcyusz Constraint Qualification \eqref{gcqc}, then the set of Lagrange multipliers at $\bar{x}$ is non-empty.  
\end{theorem}

In general the Lagrange multiplier at $\bar{x}$ is not unique.
In order to obtain uniqueness,
we have to strengthen the condition \eqref{gcqc}.
In fact,
the strict Robinson-Zowe-Kurcyusz condition at $\bar{x}$,
namely
\begin{equation} \label{sgcqc}
	Y=g'(\bar{x})X-(\mathcal{R}_{K}(g(\bar{x}))\cap\bar{\lambda}^{\perp}),
\end{equation}
implies 
that the multiplier $\bar{\lambda}$ is the unique Lagrange multiplier at $\bar{x}$ for Problem \eqref{genminpb},
see
\cite[Theorem~2.2]{SHA}.

\subsection{\textbf{Necessary and sufficient second-order optimality conditions}}
We recall
necessary and sufficient second-order optimality conditions for the general optimization problem \eqref{genminpb}
from \cite{GW}.
To state these second-order conditions,
we need the notion of the critical cone at a point $\bar{x}\in\mathcal{F}$.
 
\begin{definition} \cite[Section~5.3]{GW}
	Let $\bar{x}$ be a feasible point of Problem \eqref{genminpb}, i.e., $\bar{x} \in \mathcal{F}$.
	The critical cone $C(\bar{x})$ at $\bar{x}$ is defined via
	\begin{equation} \label{genccindmul}
		C(\bar{x}):=\{\delta x \in X;g'(\bar{x})\delta x \in \mathcal{T}_{K}(g(\bar{x})) \text{ and }f'(\bar{x})\delta x \leq 0 \}.
	\end{equation}   
\end{definition}

The critical cone $C(\bar{x})$ is a closed and convex cone.
If $\bar{x}$ is a stationary point,
it is straightforward to check
\begin{equation} \label{ccindmul}
	C(\bar{x})=\{\delta x \in X;g'(\bar{x})\delta x \in \mathcal{T}_{K}(g(\bar{x})) \text{ and }f'(\bar{x})\delta x = 0 \}.
\end{equation}

As a final ingredient,
we need the notion of a Legendre form.
\begin{definition}
	\label{def:legendre} \cite[p.184]{FreBo}
	Let $a \colon X \times X$ be a bounded bilinear form.
	We say that $a$ is a Legendre form
	if
	$h \mapsto a(h,h)$ is sequentially weakly lower semicontinuous
	and
	\begin{equation*}
		h_k \weakly h \text{ and } a(h_k,h_k) \to a(h,h)
		\quad\Rightarrow\quad
		h_k \to h
	\end{equation*}
	holds for all sequences $\{h_k\} \subset X$.
\end{definition}

By $L''$
we denote the second-order partial derivative of the Lagrangian $L$ w.r.t.\ $x$,
i.e.,
\begin{equation*}
	L''(\bar{x},\bar{\lambda})[\delta x_1,\delta x_2]
	=
	f''(\bar{x})[\delta x_1,\delta x_2]
	+
	\langle \bar\lambda,g''(\bar{x})[\delta x_1,\delta x_2]\rangle _{Y^{\star},Y}
\end{equation*}
for all $\delta x_1, \delta x_2 \in X$.
We are now in position to state necessary or sufficient optimality conditions for Problem \eqref{genminpb}.

\begin{theorem}\label{theo:CSN} \cite[Theorem~5.7]{GW}
	Let $(\bar x, \bar\lambda)$ be a KKT point
	of Problem \eqref{genminpb}.
	Further, we assume that
	\begin{enumerate}
		\item
			$X$ is a Hilbert space,
		\item
			$K$ is polyhedric at $(g(\bar x), \bar\lambda)$,
		\item
			the strict Robinson-Zowe-Kurcyusz condition \eqref{sgcqc} holds,
		\item
			$L''(\bar x, \bar\lambda) : X \times X \to \R$
			is a Legendre form.
	\end{enumerate}
	Then, the following holds.
	\begin{enumerate}[label=(\roman*)]
		\item
			If $\bar x$ is a local minimizer,
			then
			\begin{equation}
				\label{wgcqc1}
				L''(\bar{x},\bar{\lambda})[\delta x,\delta x]
				% =
				% f''(\bar{x})[\delta x,\delta x]+\langle \bar\lambda,g''(\bar{x})[\delta x,\delta x]\rangle _{Y^{\star},Y}
				\geq
				0
				\qquad\forall \delta x \in C(\bar x).
			\end{equation}
		\item
			If
			\begin{equation}
				\label{wgcqc2}
				L''(\bar{x},\bar{\lambda})[\delta x,\delta x]
				% =
				% f''(\bar{x})[\delta x,\delta x]+\langle \bar\lambda,g''(\bar{x})[\delta x,\delta x]\rangle _{Y^{\star},Y}
				>
				0
				\qquad\forall \delta x \in C(\bar x) \setminus \{0\}
			\end{equation}
			then $\bar x$ is a local minimizer
			and the quadratic growth condition
			\begin{equation} \label{qgcd}
				f(x)\geq f(\bar{x})+ \alpha\lVert x-\bar{x}\rVert_{X}^{2}
				\qquad\forall x \in \mathcal{F} \cap \bar{B}(\bar x;\varepsilon)
			\end{equation}   
			is satisfied for some constants $\alpha,\varepsilon > 0$.
	\end{enumerate}
\end{theorem}
Note that the gap between the necessary condition \eqref{wgcqc1}
and the sufficient condition \eqref{wgcqc2}
is as small as possible.
If one is only interested in necessary or sufficient conditions,
then some of the assumptions can be relaxed,
see
\cite[Theorems~5.4, 5.6]{GW}.
 
\section{\textbf{Verification of the conditions for the application of the general theory}}
\label{sec:soc}
In this section,
we verify the assumptions of Theorem~\ref{theo:CSN}
for our optimal control problem \eqref{OCP}.
First,
we simplify the problem by removing the dependence on $\tau$
and verify the constraint qualification.
Afterwards, this result is utilized to prove
the constraint qualification for Problem \eqref{OCP}.
Finally, we check that the second derivative of the Lagrangian
is a Legendre form.

\subsection{\textbf{A problem without \texorpdfstring{$\tau$}{tau}-dependence}}
We study a simplified constraint
in which we remove the $\tau$-dependence.
That is,
we introduce the constraint mapping
\begin{equation} \label{defopG}
G: H^1(0,T;\mathbb{R}^2) \rightarrow H^1(0,T): \gamma \mapsto g\circ\gamma
\end{equation}
and consider the constraint
\begin{equation}
	\label{eq:constraint_simplified}
	G(\gamma) \in K,
\end{equation}
where $K$ is the closed convex cone
\begin{equation} \label{introdeK}
	K
	:=
	H^1(0,T)_-
	:=
	\set{v \in H^1(0,T) ;\ v \leq 0}
	\subset
	H^1(0,T)
	.
\end{equation}
Note that
we get
\begin{equation*}
	\Uad
	% =\set{ \gamma \in H^1(0,T;\mathbb{R}^2);\ 
	% (g \circ\gamma)(t)\leq 0 \text{, for all } t\in[0,T]}
	% =\set{ \gamma \in H^1(0,T;\mathbb{R}^2);\ 
	% 	g \circ\gamma \in K}
	=\set{ \gamma \in H^1(0,T;\mathbb{R}^2);\ 
	G(\gamma) \in K},
\end{equation*}
by the definition \eqref{eq:setadmcdtsb} of $\Uad$,
i.e.,
\eqref{eq:constraint_simplified}
encodes the constraint on $\gamma$
in our control problem \eqref{OCP}.

\begin{proposition}
\label{GfD}
The mapping $G$ from \eqref{defopG}
is twice continuously differentiable
and
the derivatives are given by
\begin{subequations}
	\begin{align}
		DG(\gamma) \delta\gamma
		&=
		% \frac{\partial g}{\partial x_{1}}\circ\gamma
		% \ \delta\gamma_{1}+\frac{\partial g}{\partial x_{2}}\circ\gamma\ \delta
		% \gamma_{2},
		g'(\gamma) \delta\gamma
		&&
		\forall\delta\gamma\in H^{1}(0,T;\mathbb{R} ^{2})
		,
		\label{0ba}%
		\\
		D^{2}G(\gamma) [\delta\gamma,\widetilde{\delta\gamma}]
		&=
		\delta\gamma^\top \nabla^2 g(\gamma) \widetilde{\delta\gamma}
		&&
		\forall
		\delta\gamma, \widetilde{\delta\gamma} \in H^{1}(0,T;\mathbb{R} ^{2})
		.
		\label{21ba}%
	\end{align}
	Here, $g'$ and $\nabla^2 g$ denote the Jacobian and the Hessian of $g$, respectively.
\end{subequations}
\end{proposition}
\begin{proof}
	It is well known that a function belongs to $H^1(0,T)$
	if and only if it is absolutely continuous with derivative in $L^2(0,T)$.
	In this case, the classical derivative
	(which exists almost everywhere)
	coincides with the weak derivative.
	Consequently,
	the chain rule implies
	\begin{equation*}
		\nabla( G(\gamma) )
		=
		g'(\gamma) . \nabla\gamma
		\qquad\forall \gamma \in H^1(0,T;\R^2),
	\end{equation*}
	where $\nabla$ denotes the weak derivative with respect to the $t$ variable.
	This shows that $G$ is well defined as a mapping from $H^1(0,T;\R^2)$ to $H^1(0,T)$.
	In order to finish the proof,
	it is sufficient to check that the mappings
	\begin{align*}
		H^1(0,T;\R^2) \ni \gamma &\mapsto g(\gamma) \in L^2(\Omega)
		\\
		H^1(0,T;\R^2) \ni \gamma &\mapsto g'(\gamma) \in L^\infty(\Omega)
	\end{align*}
	are twice continuously differentiable.
	This, however,
	follows from the embedding
	$H^1(0,T;\R^2) \hookrightarrow L^\infty(0,T;\R^2)$,
	$g \in C^3(\R^2)$
	and \cite[Theorem~9]{GKT}.
\end{proof}

We are going to verify that the constraint qualification
\eqref{gcqc}
is satisfied in all feasible points of \eqref{eq:constraint_simplified}.
\begin{lemma}
	\label{some_lemma}
	There exists $\delta > 0$
	such that
	\begin{equation}
		\begin{aligned}
		H^{1}(0,T)
		&=
		DG(\bar{\gamma})H^{1}(0,T;\mathbb{R} ^{2})
		\\&\qquad
		-
	\Big\{w\in H^{1}(0,T);w(t)=0,\forall t\in\lbrack0,T]:g(\bar\gamma(t))  >-\frac{\delta}{2}\Big\}.
		\end{aligned}
	\label{18cq}%
\end{equation}
holds for all feasible $\bar{\gamma}\in H^{1}(0,T;\mathbb{R} ^{2})$, i.e. $G(\bar\gamma)  \in K$.
\end{lemma}
\begin{proof}
	Owing to the assumptions on $g$,
	there exist
	$\varepsilon, \delta>0$ such that
	$|\nabla g(\bar\gamma(t))  |\geq\varepsilon$ for all
	$t \in [0,T]$ such that $g(\bar\gamma(t))  \geq-\delta$.

	Next, we introduce
	\begin{align*}
		A&:=\{t\in[0,T];(g\circ\bar{\gamma})(t)\leq-\delta\},
		&
		B&:=\Big\{t\in[0,T];(g\circ\bar{\gamma})(t)\geq-\frac{\delta}{2}\Big\},
	\end{align*}
	which are
	closed subsets of the time interval $[0,T]$.
	Due to the uniform continuity of $g \circ \bar\gamma$ on $[0,T]$,
	the sets $A$ and $B$ have a positive distance.
	Consequently,
	we can choose a Lipschitz function $\varphi : [0,T] \to [0,1]$
	such that
	$\varphi = 0$ on $A$
	and
	$\varphi = 1$ on $B$.

	For an arbitrary $v \in H^1(0,T)$,
	we write
	\begin{equation*}%
		v(t)
		=
		v(t)\varphi(t)+v(t)(1-\varphi(t))
		=
		\nabla g(\bar{\gamma}(t)  )
		\cdot
		\frac{v(t)\varphi(t)\nabla g(\bar{\gamma}(t)  )}{|\nabla g(\bar{\gamma}(t))|^{2}}
		+
		v(t)(1-\varphi(t)),
	\end{equation*}
	where we interpret the fraction as $0$
	whenever $\nabla g(\bar\gamma(t)) = 0$ (which implies $\varphi(t) = 0$).
	It is clear that $1 - \varphi$ belongs to the right-most set in \eqref{18cq}.
	It is straightforward to check that
		$\frac{v(t)\varphi(t)\nabla g(\bar{\gamma}(t)  )}{|\nabla g(\bar{\gamma}(t))|^{2}}$
	belongs to $H^1(0,T);\R^2)$
	and this finishes the proof.
\end{proof}

\begin{lemma}
\label{ffbd} For each feasible $\bar{\gamma}\in H^{1}(0,T;\mathbb{R} ^{2})$, we have
\begin{equation}
\Big\{w\in H^{1}(0,T);w(t)=0,\forall t\in\lbrack0,T]:g(\bar\gamma(t))  >-\frac{\delta}{2}\Big\}\subset\mathcal{R}_{K}(G(\bar{\gamma
})). \label{23cq}%
\end{equation}
\end{lemma}

\begin{proof}
	We denote by $E$ the set on the left-hand side of \eqref{23cq}.
	We have to prove
that every $w\in E\setminus\{0\}$ can be written in the form $w=r(\psi-
g\circ\bar{\gamma})$ for some $\psi\in K = H^1(0,T)_-$ and some $r>0$. This equality is equivalent to $\psi
=\frac{w}{r}+g\circ\bar{\gamma}$.
Since $\bar{\gamma}$ is feasible, $g\circ
\bar{\gamma}\in K$. From $w\in H^{1}(0,T)\hookrightarrow C([0,T])$ we have
$\left\Vert w\right\Vert _{\infty}<+\infty$.
Moreover, $w(t)=0$ for those
$t\in\lbrack0,T]$ such that $g(\bar\gamma(t))
>-\frac{\delta}{2}$, so that in this case $\psi(t)=(g\circ\bar{\gamma}%
)(t)\leq0$. If $g(\bar\gamma(t))  \leq
-\frac{\delta}{2}$, then $\psi(t)=\frac{w(t)}{r}+(g\circ\bar{\gamma}%
)(t)\leq\frac{\left\Vert w\right\Vert _{\infty}}{r}-\frac{\delta}{2}%
\leq0$ if we choose $r=\frac{2\left\Vert w\right\Vert _{\infty}}{\delta
}.$
\end{proof}

By combining the previous two lemmas,
we obtain that the Robinson-Zowe-Kurcyusz constraint qualification
is satisfied in our setting.

\begin{corollary} \label{constqua}
For each feasible $\bar{\gamma}\in H^{1}(0,T;\mathbb{R} ^{2})$, we have:%
\begin{equation}
H^{1}(0,T)=DG(\bar{\gamma})H^{1}(0,T;\mathbb{R} ^{2})-\mathcal{R}_{K}(G(\bar{\gamma})). \label{25cq}%
\end{equation}

\end{corollary}

Next, we check that also the strengthened condition
\eqref{sgcqc}
is valid.

\begin{lemma}
	\label{another_lemma}
For each feasible $\bar{\gamma}\in H^{1}(0,T;\mathbb{R} ^{2})$ and $\bar\lambda\in\mathcal{N}_{K}(G(\bar{\gamma}))$,
we have
\begin{equation}
\Big\{w\in H^{1}(0,T);w(t)=0,\forall t\in\lbrack0,T]:g(\bar\gamma(t))  >-\frac{\delta}{2}\Big\}\subset\mathcal{R}_{K}(G(\bar{\gamma
}))\cap \bar\lambda^\perp. \label{26cq}%
\end{equation}
Consequently, \eqref{sgcqc} is satisfied.
\end{lemma}
\begin{proof}
In view of Lemma \ref{ffbd}, we have just to check that $\left\langle
\bar\lambda,w\right\rangle _{H^{1}(0,T)^{\ast},H^{1}(0,T)}=0$. By 
\eqref{23cq}, $w, -w \in \mathcal{R}_{K}(G(\bar{\gamma
}))\subset\mathcal{T}_{K}(G(\bar{\gamma}))$.
Consequently, $\bar\lambda%
\in\mathcal{N}_{K}(G(\bar{\gamma}))$
implies
$\left\langle \bar\lambda,w\right\rangle _{H^{1}(0,T)^{\ast
},H^{1}(0,T)}=0$.
\end{proof}

Consequently, the results from Section~\ref{subsec:soc}
can be applied to optimization problems
with the constraint \eqref{eq:constraint_simplified}.

\subsection{\textbf{The general case}}

To be able to apply the general theory recalled in Section~\ref{subsec:soc} to our \textbf{original optimal control problem \eqref{ROCP}}, we must previously extend our reduced cost functional (see Definition \ref{defrcf})
\begin{equation}
\hat{J}_r: H^1(0,T;\mathbb{R}^2)\times [3r,T] \to \mathbb{R}:(\gamma,\tau) \mapsto \hat{J}_r(\gamma,\tau)
\end{equation}
to a $C^2$-mapping from the entire Hilbert space $H^1(0,T;\mathbb{R}^2) \times \mathbb{R}$ into $\mathbb{R}$.
Since there is no source of confusion,
the extended functional is denoted by the same symbol.
\begin{proposition} \label{prop:exte}
The $C^2$-mapping $\hat{J}_r : H^1(0,T;\mathbb{R}^2)\times [3r,T] \to \mathbb{R}$ can be extended to a $C^2$-mapping
$\hat{J}_r : H^1(0,T;\mathbb{R}^2)\times \R \to \mathbb{R}$.
\end{proposition}
\begin{proof}
We consider a $C^\infty$ cut-off function
$\theta:\mathbb{R} \rightarrow [0,1] $ such that
\[
\theta(\tau)=
\begin{cases}
0 & \text{for }\tau\in]-\infty,3r-\frac{2}{3}\frac{T-3r}{3}%
]\cup\lbrack T+\frac{2}{3}\frac{T-3r}{3},+\infty\lbrack,\\
1 & \text{for }\tau\in\lbrack3r-\frac{1}{3}\frac{T-3r}{3}%
,T+\frac{1}{3}\frac{T-3r}{3}].
\end{cases}
\]
Then, we extend $\hat J_r$ via
\[
\hat{J}_{r}(\gamma,\tau):=
\begin{cases}
	0 & \text{if }
\tau\in ]-\infty,3r-\frac{2}{3}\frac{T-3r}{3}], \\
	\theta(\tau)\Big(6\hat{J}_{r}(\gamma,6r-\tau)-8\hat{J}_{r}(\gamma,9r-2\tau
	) \\\qquad\qquad +3\hat{J}_{r}(\gamma,12r-3\tau)\Big)
	&\text{if }\tau\in\lbrack3r-\frac{T-3r}%
{3},3r],\\
\theta(\tau)\Big(6\hat{J}_{r}(\gamma,2T-\tau)-8\hat{J}_{r}(\gamma,3T-2\tau
) \\\qquad\qquad +3\hat{J}_{r}(\gamma,4T-3\tau)\Big)
&\text{if }\tau\in\lbrack T,T+\frac{T-3r}{3}],
\\
	0 & \text{if }
\tau\in \lbrack T+\frac{2}{3} \frac{T-3r}{3},+\infty\lbrack.
	\end{cases}
\]
We already know that $\hat J_r$
is $C^2$ on
$H^1(0,T;\mathbb{R}^2)\times [3r,T]$,
see
Section~\ref{sec:prelim}.
It can be easily verified that the extended functional $\hat J_r$
is $C^2$ on
$H^{1}(0,T;\mathbb{R} ^{2})\times\mathbb{R} $.
\end{proof}

Next, we introduce the mapping
\begin{equation}\label{refopbeq1}
\mathcal{G}:H^1(0,T;\mathbb{R}^2) \times \mathbb{R} \rightarrow  H^1(0,T)\times \mathbb{R}:(\gamma,\tau) \mapsto (G(\gamma),\tau)\equiv(g\circ \gamma,\tau)
\end{equation}
and the closed convex subset $\mathcal{K}$ of the real Hilbert space $H^1(0,T)\times \mathbb{R}$ via
\begin{equation}\label{refopbeq2}
\mathcal{K}:=K \times [3r,T]=H^1(0,T)_{-}\times [3r,T]
.
\end{equation} 
By Examples \ref{exa:poly} and Proposition \ref{prop:pps}, $\mathcal{K}$ is a polyhedric set in $H^1(0,T)\times \mathbb{R}$.

The Problem \eqref{OCP}
can now be stated equivalently as
\begin{equation} \label{ourminpb}
	\text{minimize} \  \hat{J}_{r}(\gamma,\tau) \text{ subject to } \mathcal{G}(\gamma,\tau)\in \mathcal{K}.
\end{equation}
The Lagrangian function
associated to our minimization problem \eqref{ourminpb} is given by
\begin{equation}
	\begin{aligned}
		L((\gamma,\tau),(\lambda,\mu))&=\hat{J}_{r}(\gamma,\tau)+\left\langle (\lambda,\mu),\mathcal{G}(\gamma,\tau)\right\rangle=\hat{J}_{r}(\gamma,\tau)
		+\left\langle \lambda,G(\gamma)\right\rangle +\left\langle
		\mu,\tau\right\rangle \\ &=\hat{J}_{r}%
		(\gamma,\tau)+\left\langle \lambda,g\circ\gamma\right\rangle +\left\langle
		\mu,\tau\right\rangle ,
	\end{aligned}
\label{31hcq}%
\end{equation}
for all
$(\gamma,\tau) \in H^{1}(0,T;\mathbb{R}^2)\times \mathbb{R}$ and $(\lambda,\mu) \in (H^{1}(0,T)\times \mathbb{R})^{\star}=H^{1}(0,T)^{\star} \times \mathbb{R}^{\star}.$
 
Let $(\bar{\gamma},\bar{\tau}) \in H^1(0,T;\mathbb{R}^2) \times \mathbb{R}$ be feasible, i.e., $\mathcal{G}(\bar{\gamma},\bar{\tau})\in \mathcal{K}$.
By \eqref{25cq}, we have
\[
H^{1}(0,T)=DG(\bar{\gamma})H^{1}(0,T;\mathbb{R} ^{2})-\mathcal{R}_{K}(G(\bar{\gamma}))
.
\]
Taking into account that $D\mathcal{G}\left(  \bar{\gamma},\bar{\tau
}\right)  =DG(\bar{\gamma})\otimes I_{\mathbb{R} }$
and using
Proposition \ref{rcpr},
it follows
\begin{equation}
H^{1}(0,T)\times\mathbb{R} =D\mathcal{G}\left(  \bar{\gamma},\bar{\tau}\right)  (H^{1}(0,T;\mathbb{R} ^{2})\times\mathbb{R} )-\mathcal{R}_{\mathcal{K}}(\mathcal{G}\left(  \bar{\gamma},\bar{\tau}\right)
).\label{31cq}%
\end{equation}
This shows that a constraint qualification is also satisfied at any feasible
point $\left(  \bar{\gamma},\bar{\tau}\right)  $ for our optimal control
problem \eqref{ourminpb}. By Theorem \ref{exLagMult}, we thus get the existence of Lagrange
multipliers at any feasible point $\left(  \bar{\gamma},\bar{\tau}\right)  $ satisfying the first order conditions  \eqref{varineq}, \eqref{varineq2}.
Thus, there exists thus an ordered pair $\left(  \bar{\lambda},\bar{\mu}\right)
\in\mathcal{N}_{\mathcal{K}}\left(  \mathcal{G}\left(  \bar{\gamma},\bar{\tau
}\right)  \right)=\mathcal{N}_{K}(G(\bar{\gamma}))%
\times\mathcal{N}_{[3r,T]}(\bar{\tau})$ such that $
\hat{J}_{r}^{\prime}\left(  \bar{\gamma},\bar{\tau}\right)
+(G^{\prime}(\bar{\gamma})^{\ast}\bar{\lambda},\bar{\mu})=(0_{H^1(0,T;\mathbb{R}^2)^{\ast}},0_{\mathbb{R} ^{\ast}})$
and this is equivalent to
\begin{subequations}
	\label{31dcq}
	\begin{align}
		\label{31dcq:1}
		D_{\gamma}\hat{J}_{r}\left(  \bar{\gamma}%
		,\bar{\tau}\right)  +G^{\prime}(\bar{\gamma})^{\ast}\bar{\lambda}&=0_{H^1(0,T;\mathbb{R}^2)^{\ast}%
		}\ ,\\
		\label{31dcq:2}
		D_{\tau}\hat{J}_{r}\left(  \bar{\gamma},\bar
		{\tau}\right)  +\bar{\mu}&=0_{\mathbb{R} ^{\ast}}\ .
	\end{align}
\end{subequations}
The fact that $\bar\lambda\in\mathcal{N}_{K}\left(  G(\bar\gamma)  \right)$ implies the following.
\begin{proposition} \label{radmeas}
The multiplier $\bar\lambda$ is a positive Radon measure on the compact interval $[0,T]$.
\end{proposition}
\begin{proof}
	We have the continuous embedding $H^1(0,T) \hookrightarrow C([0,T])$
	and this embedding has a dense range.
	Due to $K \subset \TT_K(G(\bar\gamma))$,
	one can check that $\bar\lambda$ takes non-negative values on
	$H^{1}(0,T)_+$.
	Via a standard argument (see, e.g., \cite[Lemma~6.53]{BoSh})
	one verifies that $\bar\lambda$ is continuous
	w.r.t.\ the supremum norm.
	Consequently, it can be extended (uniquely)
	to a continuous and non-negative functional on $C([0,T])$.
	The Riesz representation theorem implies the claim,
	see \cite[Theo. 2.14 p.42]{WRud}.
\end{proof}

In conclusion, $\bar{\lambda}$ is a positive Radon measure satisfying $\langle \bar{\lambda},G(\bar{\gamma}\rangle =0$ i.e. $\langle \bar{\lambda},g \circ \bar{\gamma}\rangle =0$ and verifying equation \eqref{31dcq:1}.

Next,
we check that also the strict Robinson-Zowe-Kurcyusz condition is satisfied.
Indeed,
Lemma~\ref{some_lemma}
and Lemma~\ref{another_lemma}
imply
\[
	H^{1}(0,T)=DG(\bar{\gamma})H^{1}(0,T;\mathbb{R} ^{2})-\mathcal{R}_{K}(G(\bar{\gamma}))\cap \bar{\lambda}^\perp
	.
\]
On the other hand,
\begin{equation}\mathbb{R} =\mathbb{R} -\mathcal{R}_{[3r,T]}(\bar{\tau})\cap\bar{\mu}^\perp\label{31fcq}%
\end{equation}
is clear.
Consequently,
\begin{equation}
H^{1}(0,T)\times\mathbb{R} =D\mathcal{G}\left(  \bar{\gamma},\bar{\tau}\right)  (H^{1}(0,T;\mathbb{R} ^{2})\times\mathbb{R} )-\mathcal{R}_{\mathcal{K}}(\mathcal{G}\left(  \bar{\gamma},\bar{\tau}\right)
)\cap(\bar{\lambda},\bar{\mu})^\perp.\label{31gcq}%
\end{equation}
This
shows that the strict Robinson-Zowe-Kurcyusz condition is satisfied at
$\left(  \bar{\gamma},\bar{\tau}\right)  $ with the multiplier $(\bar\lambda, \bar\mu)$.
Consequently, the Lagrange multiplier at $(\bar{\gamma},\bar{\tau})$ is unique.

Finally,
we show that the multiplier $\bar\lambda$ enjoys better regularity.

\begin{proposition}\label{prop:lmsqi}
Let us assume that $(\bar{\gamma},\bar{\tau})$ is a local minimizer of our problem \eqref{ourminpb}
and let us denote by $(\bar{\lambda},\bar{\mu})$ the unique Lagrange multiplier at $(\bar{\gamma},\bar{\tau})$.
Then $\bar{\lambda} \in L^2(0,T).$  
\end{proposition}

\begin{proof}
	We follow the arguments in the proof of Theorem \ref{regoptraj}.
	We will see that we can pass to the limit in the optimality system
	\eqref{add22}-\eqref{add24} and this will imply the higher regularity of $\bar\lambda$.
By inequality \eqref{add35} in the fourth step of the proof
of Theorem \ref{regoptraj},
we have that for large $\kappa$
\begin{equation}
\left\Vert 2\kappa\left[  g\left(  \tilde{\gamma}_{\kappa}(\cdot)\right)
\right]  _{+}\right\Vert _{L^{2}(0,T)}\lesssim(1+\left\Vert \tilde{\gamma
}_{\kappa}^{\prime\prime}\right\Vert _{L^{2}(0,T;
\mathbb{R}^{2})})
\end{equation}
and by \eqref{add42} of the same proof, the norms  $\left\Vert \tilde{\gamma}_{\kappa}^{\prime\prime
}\right\Vert _{L^{2}(0,T;\mathbb{R} ^{2})}$ are uniformly bounded. This implies that the $L^{2}(0,T)$-norms of
$\left(  2\kappa\left[  g\left(  \tilde{\gamma}_{\kappa}(\cdot)\right)
\right]  _{+}\right)  _{\kappa\geq1}$ are uniformly bounded.
Thus, the sequence
$\left(  2\kappa\left[  g\left(  \tilde{\gamma }_{\kappa}(\cdot)\right)  \right]  _{+}\right)_{\kappa\geq 1}$
possesses  a subsequence
which converges weakly
to some element $\tilde{\lambda}\in L^{2}(0,T)$. We also have seen
at the end of the proof of Theorem \ref{regoptraj} that $\left(  \tilde{\gamma}_{\kappa
}^{\prime\prime}\right)_{\kappa\geq1}$ converges weakly to $\bar{\gamma
}^{\prime\prime}$ in $L^{2}(0,T;\mathbb{R}^{2})$.
Using the fact that $\left(\tilde{\gamma}_{\kappa}\right)_{\kappa\geq 1}$ tends to $\bar{\gamma}$ in $C([0,T];\mathbb{R}^2)$,
we are able to pass to the limit in the differential equation \eqref{add22}, and we obtain
\begin{equation}\label{eqtild}
\lambda_{\gamma}\bar{\gamma}^{\prime\prime}(t)=P_{(\bar{\gamma},\bar{\tau})}(t)+\tilde
{\lambda}(t)\nabla g(\bar\gamma(t)),
\end{equation}
 in $L^{2}(0,T;\mathbb{R}^{2})$, thus for a.e. $t \in [0,T]$, where by \eqref{add20}
\begin{equation} \label{add43}
P_{(\bar{\gamma},\bar{\tau})}(t)=\int_{\Gamma_{1}}c_{R}e^{-\frac{2}{R^{2}}\left\vert
x-\bar{\gamma}(t)\right\vert ^{2}}(x-\bar{\gamma}(t))\ p_{(\bar{\gamma},\bar{\tau})}%
(x,t) \d S(x),
\end{equation}
and
$p_{(\bar{\gamma},\bar{\tau})}$ denotes the adjoint state corresponding to $(\bar{\gamma},\bar{\tau})$ defined by \eqref{varformofAdjpb}.
By \eqref{varineq} and equation \eqref{eqtild}, we obtain by integrating by parts using
$\bar{\gamma}\in H^{2}(0,T;\mathbb{R} ^{2})$ and $\bar{\gamma}^{\prime}(T)=\bar{\gamma}^{\prime}(0)=0_{\mathbb{R} ^{2}}$
\begin{equation} \label{add44}
	\begin{aligned}
D_{\gamma}\hat{J}_{r}\left(  \bar{\gamma},\bar{\tau}\right)  \delta\bar
{\gamma}
&=\int_{0}^{T}P_{(\bar{\gamma},\bar{\tau})}(t) \cdot \delta\bar{\gamma}(t) \d t+\lambda
_{\gamma}\int_{0}^{T}\bar{\gamma}'(t) \cdot \delta\bar{\gamma}'(t) \d t\\
&=\lambda_{\gamma}\int_{0}^{T}\bar{\gamma}^{\prime\prime}(t) \cdot \delta\bar{\gamma
}(t) \d t-\int_{0}^{T}\tilde{\lambda}(t)\nabla g(\bar\gamma(t))
 \cdot \delta\bar{\gamma}(t) \d t
 \\&\qquad+\lambda_{\gamma}\int_{0}^{T}\bar{\gamma}'%
(t) \cdot \delta\bar{\gamma}'(t) \d t\\
&
=-\int_{0}^{T}\tilde{\lambda}(t)\nabla g(\bar\gamma(t))
 \cdot \delta\bar{\gamma}(t) \d t
=-\int_{0}^{T}\tilde{\lambda}(t) g'(\bar\gamma(t)) \delta\bar{\gamma}(t) \d t
 .
	\end{aligned}
\end{equation}
Now, let us identify $G^{\prime}(\bar\gamma)  ^{\ast}%
\tilde{\lambda}$. First, by \eqref{0ba} we have $DG(\bar{\gamma})\delta\bar{\gamma
}= g'(\bar\gamma) \delta\bar\gamma$.
Consequently,
\begin{equation} \label{add45}
\left\langle G^{\prime}(\bar\gamma)  ^{\ast}\tilde{\lambda
},\delta\bar{\gamma}\right\rangle =\left\langle \tilde{\lambda},DG(\bar{\gamma
})\delta\bar{\gamma}\right\rangle =\left\langle \tilde{\lambda},
g'(\bar{\gamma}) \delta\bar{\gamma}\right\rangle _{H^{1}%
(0,T)^{\ast}, H^{1}(0,T)}
\end{equation}
for all $\delta \bar{\gamma} \in H^1(0,T;\mathbb{R}^2).$
By combining equations \eqref{add44} and \eqref{add45}, we have
\begin{equation} \label{eqmult2}
D_{\gamma}\hat{J}_{r}\left(  \bar{\gamma},\bar{\tau}\right)  +G^{\prime
}(\bar\gamma)  ^{\ast}\tilde{\lambda}=0_{H^{1}(0,T;\mathbb{R} ^{2})^{\ast}}.
\end{equation}
On the other hand, the multiplier $\bar\lambda$
is the solution of
\begin{equation} \label{eqmult1}
	\bar\lambda \in \NN_K(G(\bar\gamma)),
	\qquad
D_{\gamma}\hat{J}_{r}\left(  \bar{\gamma},\bar{\tau}\right)  +G^{\prime
}(\bar\gamma)  ^{\ast}\bar{\lambda}=0_{H^{1}(0,T;\mathbb{R} ^{2})^{\ast}}.
\end{equation}
Thus, it remains to show
$\tilde\lambda \in \NN_K(G(\bar\gamma)) = K^\circ \cap G(\bar\gamma)^\perp$,
where the last identity follows from $K$ being a cone.

Since
$\tilde{\lambda}\in L^{2}(0,T)$ is the weak limit in $L^2(0,T)$ of a subsequence of the sequence $\left(  2\kappa\left[  g\left(  \tilde{\gamma
}_{\kappa}(\cdot)\right)  \right]  _{+}\right)  _{\kappa\geq1}$, $\tilde{\lambda}$ is non-negative almost everywhere and thus $\tilde{\lambda} \in K^{\circ}$ since $K=H^1(0,T)_{-}$.
It remains to prove that $\langle \tilde{\lambda},G(\bar{\gamma})\rangle =0$. 

As $G(\bar{\gamma}) \in K$, we know already that $\langle \tilde{\lambda},G(\bar{\gamma})\rangle \leq 0$ so that we must now prove that $\langle \tilde{\lambda},G(\bar{\gamma})\rangle \geq 0$.
We have that
$\tilde{\lambda}$ is the weak limit
of a subsequence
$\left(  2\kappa\left[  g\left(  \tilde{\gamma }_{\kappa_l}(\cdot)\right)  \right]  _{+}\right)_{l\geq 1}$
in
$L^{2}(0,T)\subset H^{1}(0,T)^{\ast}$.
Further,
$( G(\tilde\gamma_\kappa) )_{\kappa \ge 1}$
converges in $L^2(0,T)$
towards
$G(\bar\gamma)$.
Consequently,
\[%
\langle \tilde{\lambda},G(\bar\gamma)  \rangle
=\lim\limits_{l\rightarrow+\infty}\int_{0}^{T}2\kappa_{l}\left[  g(
\tilde{\gamma}_{\kappa_{l}}(t))\right]  _{+}\ g(\tilde\gamma_{\kappa_l}(t))\d t
\ge 0
.
\]%
In conclusion, $\langle \tilde{\lambda},G(\bar\gamma)  \rangle =0$, so that $\tilde{\lambda} \in \mathcal{T}_{K}(G(\bar{\gamma}))^{\circ}.$
With \eqref{eqmult2}, this implies that $(\tilde{\lambda},\bar{\mu})$ is also a Lagrange multiplier at $(\bar{\gamma},\bar{\tau})$ so that by uniqueness of the Lagrange multiplier we get $\bar{\lambda}=\tilde{\lambda}.$
Consequently $\bar{\lambda} \in L^2(0,T).$   
\end{proof}

The above proof shows that
$\bar{\lambda}$ is the weak limit of $( 2\kappa \left[ g\circ  \tilde{\gamma}_{\kappa}  \right]  _{+})_{\kappa \ge 1}$
in $L^{2}(0,T)$ where $(\tilde{\gamma}_{\kappa},\tilde{\tau}_{\kappa})$ is a minimizer
of the functional $\tilde{J}_{\kappa}$ \eqref{add2} on the set $C$ \eqref{add3}. 

\subsection{\textbf{The second derivative is a Legendre form}}
The next goal is to verify
that
if $(\bar \gamma,\bar\tau)$ is a stationary point (see Definition \ref{statpoint}) with Lagrange multiplier $(\bar\lambda,\bar\mu) $,
that the quadratic form associated to the second-derivative of the Lagrangian 
\begin{equation*}
D^2_{(\gamma,\tau)} L(\bar \gamma,\bar\tau,\bar\lambda,\bar\mu): H^1(0,T;\mathbb{R}^2) \times \mathbb{R} \rightarrow \mathbb{R}: (\delta\gamma,\delta\tau) \mapsto D^2_{(\gamma,\tau)} L(\bar \gamma,\bar\tau,\bar\lambda,\bar\mu)[\delta\gamma,\delta\tau]^2
\end{equation*}
is a Legendre form, see Definition \ref{def:legendre}.

\textbf{Notation}: In the following, if $f$ is a $C^{2}$ function from an open set $A$ of a
Banach space $X$ into a Banach space $Y$, for $x\in A$ and $h\in X$, we will
shortly write $D^{2}f(x)[h]^{2}$ for $D^{2}f(x)[h,h]$.

For convenience,
we recall from \eqref{31hcq}
that the Lagrangian is given by
\begin{equation}
L((\gamma,\tau),(\lambda,\mu))=\hat{J}_{r}%
(\gamma,\tau)+\left\langle \lambda,G(\gamma)\right\rangle +\left\langle
\mu,\tau\right\rangle ,\label{31hcq2}%
\end{equation}
where $G:H^1(0,T;\mathbb{R}^2) \to H^1(0,T): \gamma \mapsto g\circ\gamma$, so that by \cite[p.95]{HCar}:
\begin{equation} \label{31hcq2sec}
	D^2_{(\gamma,\tau)} L(\bar \gamma,\bar\tau,\bar\lambda,\bar\mu)[\delta\gamma,\delta\tau]^2=D^2_{(\gamma,\tau)} \hat{J}_{r}(\bar \gamma,\bar\tau)[\delta\gamma,\delta\tau]^2+\left\langle \bar\lambda,D^2 G(\bar\gamma)[\delta\gamma]^2\right\rangle.
\end{equation}
Here,
$(\bar\lambda,\bar\mu)$ denotes the Lagrange multiplier at $(\bar\gamma,\bar\tau)$.

First, we study the term involving $\bar\lambda$.
\begin{lemma} \label{LegForm1}
 The mapping
	\begin{align*}
H^1(0,T;\mathbb{R}^2) \rightarrow \mathbb{R}:\delta\gamma
&\mapsto  \langle \bar\lambda,D^2 G(\bar\gamma)[\delta\gamma]^2\rangle _{H^1(0,T)^{\star},H^1(0,T)}
\\
&\qquad=\int_{[0,T]} \big(D^2 G(\bar\gamma)[\delta\gamma]^2\big) (t)\d\bar \lambda(t) 
	\end{align*}
is sequentially weakly continuous, i.e., sequentially continuous when $H^1(0,T;\mathbb{R}^2)$ is endowed with the topology of the weak convergence.
\end{lemma}
\begin{proof}
By \eqref{21ba}:
\[
\int_{[0,T]}\big(D^{2}G(\bar\gamma)  [\delta
\gamma]^{2}\big)(t)\d\bar{\lambda}(t)
=\int_{[0,T]}
\delta\gamma(t)^\top \nabla^2 g(\bar\gamma(t)) \delta\gamma(t)
\d\bar{\lambda}(t).
\]
If the sequence $\left(  \delta\gamma^{(n)}\right)  _{n\in\mathbb{N} }$ converges weakly to $\delta\gamma$ in $H^{1}(0,T;\mathbb{R} ^{2})$, then due to the compact embedding from $H^{1}(0,T;\mathbb{R} ^{2})$ into $C([0,T];\mathbb{R} ^{2})$, the sequence $\left(  \delta\gamma^{(n)}\right)  _{n\in\mathbb{N} }$ converges also strongly to $\delta\gamma$ in $C([0,T];\mathbb{R} ^{2})$.
The regularity $g\in C^{3}(\mathbb{R} ^{2})$
implies that
$\nabla^2 g(\bar\gamma) \in C([0,T]; \R^{2 \times 2})$.
Consequently,
$ (\delta\gamma^{(n)})^\top \nabla^2 g(\bar\gamma) \delta\gamma^{(n)} $
converges strongly in $C([0,T])$ towards
$ \delta\gamma^\top \nabla^2 g(\bar\gamma) \delta\gamma $.
Thus $\bar{\lambda}$ being a (positive) Radon measure on $[0,T]$ by Proposition \ref{radmeas}, i.e., a
continuous linear form on $C([0,T])$,
implies
\[
	\int_{[0,T]}\big(D^{2}G(\bar\gamma)  [\delta \gamma^{(n)}]^{2}\big)(t)\d\bar{\lambda}(t)
\to
\int_{[0,T]}\big(D^{2}G(\bar\gamma)  [\delta \gamma]^{2}\big)(t)\d\bar{\lambda}(t)
\]
as $n\rightarrow+\infty$,
where we used again \eqref{21ba}.
\end{proof} 

Next,
we study the second derivative of the reduced objective $\hat J_r$
which appears in \eqref{31hcq2sec}.
Owing to the structure of this second derivative,
see
\eqref{DgghatJr}, \eqref{DtghatJr} and \eqref{sectideriv},
the proof is divided into three lemmas.
\begin{lemma} \label{LegForm2}
	Let $(\bar\gamma,\bar\tau)\in H^1(0,T;\mathbb{R}^2) \times [3r,T]$ be given.
\begin{enumerate}
\item
The mapping
\begin{equation*}
	H^1(0,T;\mathbb{R}^2) \rightarrow \mathbb{R}: \delta\gamma \mapsto D_{\gamma\gamma}  \hat J_r(\bar\gamma,\bar\tau)\cdot[\delta\gamma]^2-\lambda_\gamma | \delta\gamma |_{H^1(0,T;\mathbb{R}^2)}^2
\end{equation*}
is weakly sequentially continuous.
\item
Consequently, the mapping $H^1(0,T;\mathbb{R}^2) \rightarrow \mathbb{R}: \delta\gamma \mapsto D_{\gamma\gamma}  \hat J_r(\bar\gamma,\bar\tau)\cdot[\delta\gamma]^2$ is weakly sequentially lower semicontinuous.
\end{enumerate}
\end{lemma}
\begin{proof}
(1)
Using \eqref{DgghatJr},
we have
\begin{equation}
	\label{eq:second_derivative}
\begin{aligned}
	\MoveEqLeft
	D_{\gamma\gamma}  \hat J_r(\bar\gamma,\bar\tau)[\delta\gamma]^2-\lambda_\gamma | \delta\gamma |_{H^1(0,T;\mathbb{R}^2)}^2
	\\
	&=
	\textstyle
	\frac{\lambda_Q}{r}\int_{\frac{\bar\tau}{2}-\frac{r}{2}}^{\frac{\bar\tau}{2}+\frac{r}{2}}\int_\Omega\Big((S(\bar\gamma)-y_Q)(x,t)(D_{\gamma\gamma}S(\bar\gamma)[\delta\gamma]^2)(x,t)+v(\bar\gamma,\delta\gamma)^2(x,t) \Big) \d x \d t
	\\
	&\quad
	\textstyle
	+\frac{\lambda_\Omega}{r}\int_{\bar\tau -r}^{\bar\tau}\int_\Omega  \Big((S(\bar\gamma)-y_\Omega)(x,t)(D_{\gamma\gamma}S(\bar\gamma)[\delta\gamma]^2)(x,t)+v(\bar\gamma,\delta\gamma)^2(x,t) \Big)  \d x \d t
	\\
	&\quad
	\textstyle
	+\int_0^T\int_\Omega\Big(  |\nabla v(\bar\gamma,\delta\gamma)(x,t)|^2 + \nabla S(\bar\gamma)(x,t).\nabla (D_{\gamma\gamma}S(\bar\gamma)[\delta\gamma]^2)(x,t)\Big)  \d x \d t ,
\end{aligned}
\end{equation}
where
$v(\bar\gamma,\delta\gamma):=DS(\bar\gamma)\delta\gamma$.

Let us consider a sequence $(\delta\gamma^n)_{n \in \mathbb{N}}$ converging weakly to $\delta\gamma$ in $H^1(0,T;\mathbb{R}^2)$.
By the compact embedding from $H^1(0,T;\mathbb{R}^2)$ into $C([0,T];\mathbb{R}^2)$,
$(\delta\gamma^n)_{n \in \mathbb{N}}$ tends strongly to $\delta\gamma$ in $C([0,T];\mathbb{R}^2)$.
The sequence of mappings   
\begin{equation*}
\bar{\Sigma}_{1} \rightarrow \mathbb{R}: (x,t) \mapsto e^{-\frac{2}{R^{2}}\left\vert x-\bar{\gamma}(t)\right\vert ^{2}%
}(x-\bar{\gamma}(t))\cdot\delta\gamma^{\left(  n\right)  }(t)
\end{equation*}
 tends to
\begin{equation*}
\bar{\Sigma}_{1} \rightarrow \mathbb{R}: (x,t) \mapsto e^{-\frac{2}{R^{2}}\left\vert x-\bar{\gamma}(t)\right\vert ^{2}}%
(x-\bar{\gamma}(t))\cdot\delta\gamma(t)
\end{equation*}
in $C(\bar{\Sigma}_{1})$
and thus a fortiori in $L^{2}\left(
\Sigma_{1}\right)  $.
Now, \eqref{diffOfG} and \eqref{ivbp} imply
\begin{equation}
	\label{eq:conv_v}
	v(\bar{\gamma},\delta\gamma^{\left(  n\right)  })
	\to
	v(\bar{\gamma},\delta\gamma)
	\qquad\text{in }
	W(0,T)
	\hookrightarrow L^2(0,T; H^1(\Omega))
	.
\end{equation}

We also have that
the function
$e^{w(\bar\gamma)}(\frac{4}{R^2}\tilde\gamma(\bar\gamma).\delta\gamma^n\tilde\gamma(\bar\gamma).\delta\gamma^n-\delta\gamma^n.\delta\gamma^n)$
converges to
$e^{w(\bar\gamma)}(\frac{4}{R^2}\tilde\gamma(\bar\gamma).\delta\gamma\tilde\gamma(\bar\gamma).\delta\gamma-\delta\gamma.\delta\gamma)$
in $C(\bar{\Sigma}_{1})$
and thus a fortiori in $L^{2}\left( \Sigma_{1}\right).$
Thus by \eqref{addivbp2},
\begin{equation}
	\label{eq:conv_DD}
	D_{\gamma\gamma}S(\bar\gamma)[\delta\gamma^n]^2
	\to
	D_{\gamma\gamma}S(\bar\gamma)[\delta\gamma]^2
	\qquad\text{in }
	W(0,T)
	\hookrightarrow L^2(0,T; H^1(\Omega))
	.
\end{equation}
Using \eqref{eq:conv_v} and \eqref{eq:conv_DD}
is enough
to pass to the limit in the integrals appearing in \eqref{eq:second_derivative}.

(2)
The mapping
$H^1(0,T;\mathbb{R}^2) \rightarrow \mathbb{R}: \delta \gamma \mapsto \lvert \delta\gamma \rvert_{H^1(0,T;\mathbb{R}^2)}^2$
is continuous and convex and thus sequentially weakly lower semicontinuous.
Thus, (2) follows from (1).
\end{proof}
  
\begin{lemma} \label{LegForm3}
The map $H^{1}(0,T;%
\mathbb{R}^{2})\times \mathbb{R}
\ni \left(  \delta\gamma,\delta\tau\right)  \mapsto D_{\gamma\tau}\hat{J}%
_{r}\left(  \bar{\gamma},\bar{\tau}\right)  [\delta\gamma,\delta\tau]$ is
a weakly sequentially continuous function from $H^{1}(0,T;\mathbb{R} ^{2})\times\mathbb{R} $ into $\mathbb{R} $. 
\end{lemma}
\begin{proof}
	We argue similarly as in the proof of the previous lemma.
Let us consider a sequence $(\delta\gamma^n,\delta\tau^n)_{n \in \mathbb{N}}$
converging weakly to $(\delta\gamma,\delta\tau)$ in $H^1(0,T;\mathbb{R}^2)\times \mathbb{R}$.
As we have seen in the proof of the preceding lemma, this implies that $v(\bar{\gamma},\delta\gamma^{\left(  n\right)})$
converges strongly to $v(\bar{\gamma},\delta\gamma)$ in $W(0,T)$.
Now, we use that $W(0,T)$ is continuously embedded into
into $C([0,T];L^2(\Omega))$
and
this implies the convergence in $L^2(\Omega)$  of
$\big{(}v(\bar\gamma,\delta\gamma^n)(.,s)\big{)}_{n \in \mathbb{N}}$
to
$v(\bar\gamma,\delta\gamma)(.,s)$
for all fixed values of $s \in [0,T]$.
Note that these functions with
$s \in \{ \tau/2 + r/2, \tau/2 - r/2, \tau, \tau - r \}$
appear in \eqref{DtghatJr}.
This allows to pass to the limit with all the terms appearing in \eqref{DtghatJr}.
\end{proof}

\begin{lemma} \label{LegForm4}
The quadratic form $\mathbb{R} \times \mathbb{R}
\rightarrow \mathbb{R}
:\left(  \delta\tau_1,\delta\tau_2\right)  \mapsto D_{\tau\tau}\hat{J}%
_{r}\left(  \bar{\gamma},\bar{\tau}\right)  [\delta\tau_1,\delta\tau_2]$ is
a weakly sequentially continuous function from $\mathbb{R} \times \mathbb{R}$ to $\mathbb{R}$.
\end{lemma}
\begin{proof}
	This is clear since strong convergence and weak convergence are equivalent on $\R$.
\end{proof}

Now, we combine the previous four lemmas.
\begin{proposition} \label{itisLegfo}
$(\bar\lambda,\bar\mu)$ denoting the Lagrange multiplier at $(\bar\gamma,\bar\tau)$, the quadratic form 
\begin{equation*}
D^2_{(\gamma,\tau)} L(\bar \gamma,\bar\tau,\bar\lambda,\bar\mu): H^1(0,T;\mathbb{R}^2) \times \mathbb{R} \rightarrow \mathbb{R}: (\delta\gamma,\delta\tau) \mapsto D^2_{(\gamma,\tau)} L(\bar \gamma,\bar\tau,\bar\lambda,\bar\mu)[\delta\gamma,\delta\tau]^2
\end{equation*}
is a Legendre form. 
\end{proposition}

\begin{proof}
By the Lemmas \ref{LegForm1}, \ref{LegForm2} (2), \ref{LegForm3}, \ref{LegForm4} and formula  \eqref{31hcq2sec}, it is clear that $D^2_{(\gamma,\tau)} L(\bar \gamma,\bar\tau,\bar\lambda,\bar\mu)$ is sequentially weakly lower semicontinuous. 

Now, still by the Lemmas \ref{LegForm1}, \ref{LegForm3}, \ref{LegForm4}, and Lemma \ref{LegForm2} (1), it follows that the mapping 
\begin{equation*}
H^1(0,T;\mathbb{R}^2) \times \mathbb{R} \rightarrow \mathbb{R}: (\delta\gamma,\delta\tau) \mapsto D^2_{(\gamma,\tau)} L(\bar \gamma,\bar\tau,\bar\lambda,\bar\mu)[\delta\gamma,\delta\tau]^2-\lambda_\gamma | \delta\gamma |_{H^1(0,T;\mathbb{R}^2)}^2
\end{equation*}
is weakly sequentially continuous.
Consider a sequence $([\delta\gamma^n,\delta\tau^n])_{n \in \mathbb{N}}$
which converges weakly to
$[\delta\gamma,\delta\tau]$
in
$H^1(0,T;\mathbb{R}^2) \times \mathbb{R}$
with
$D^2_{(\gamma,\tau)} L(\bar \gamma,\bar\tau,\bar\lambda,\bar\mu)[\delta\gamma^n,\delta\tau^n]^2 \to D^2_{(\gamma,\tau)} L(\bar \gamma,\bar\tau,\bar\lambda,\bar\mu)[\delta\gamma,\delta\tau]^2$ as $n \to +\infty$.
Due to $\lambda_\gamma > 0$,
this gives
$| \delta\gamma^n |_{H^1(0,T;\mathbb{R}^2)}^2 \to | \delta\gamma |_{H^1(0,T;\mathbb{R}^2)}^2$ as $n \to +\infty$.
By the compact embedding from $H^1(0,T;\mathbb{R}^2)$ into $L^2(0,T;\mathbb{R}^2)$,
one has also that $\delta\gamma^n$ tends to $\delta\gamma$ in $L^2(0,T;\mathbb{R}^2)$.
Thus, $\delta\gamma^n$ converges strongly to $\delta\gamma$ in $H^1(0,T;\mathbb{R}^2)$ and $([\delta\gamma^n,\delta\tau^n])_{n \in \mathbb{N}}$ strongly to $[\delta\gamma,\delta\tau]$ in $H^1(0,T;\mathbb{R}^2) \times \mathbb{R}$.
This shows that $D^2_{(\gamma,\tau)} L(\bar \gamma,\bar\tau,\bar\lambda,\bar\mu)$ is a Legendre form. 
\end{proof}

\section{\textbf{No-gap second order optimality conditions in additive manufacturing}}
\label{sec:sec_deriv_Leg_form}

In this section,
we apply the abstract Theorem~\ref{theo:CSN}
to
the optimization problem \eqref{ourminpb},
which is a reformulation of Problem \eqref{OCP}.
Let us briefly check that the assumptions of Theorem~\ref{theo:CSN}
are satisfied.
Clearly,
$H^{1}(0,T;\mathbb{R} ^{2})\times \R $
is a Hilbert space.
The polyhedricity of the set $\mathcal{K}$ given in \eqref{refopbeq2}
follows from Example \ref{exa:poly} and Proposition \ref{prop:pps}.
The strict Robin-Zowe-Kurcyusz condition \eqref{sgcqc}
was verified in \eqref{31gcq}.
Finally,
Proposition~\ref{itisLegfo}
shows that the second derivative of the Lagrangian
is a Legendre form.

First, we formulate the second-order sufficient conditions,
which follow directly form Theorem~\ref{theo:CSN}.

\begin{theorem} \label{condsuff}
Let $\left(  \bar{\gamma},\bar{\tau}\right)  \in H^{1}(0,T;\mathbb{R} ^{2})\times\lbrack3r,T]$ be a feasible point,
i.e., $\bar{\gamma}  \in \Uad$ or, equivalently, $G(\bar\gamma)  \in K$,
see \eqref{eq:setadmcdtsb}, \eqref{defopG}, \eqref{introdeK}.
Let us suppose that $\left(  \bar{\gamma
},\bar{\tau}\right)  $ satisfies the first order necessary conditions \eqref{varineq}
and \eqref{varineq2}. Let us denote by $\left(  \bar{\lambda},\bar{\mu}\right)  $ the
unique Lagrange multiplier at $\left(  \bar{\gamma},\bar{\tau}\right)  $, and
let us assume that
\begin{equation} \label{prdumin}
D^2_{(\gamma,\tau)} L(\bar \gamma,\bar\tau,\bar\lambda,\bar\mu)[\delta\gamma,\delta\tau]^2
=
D^2_{(\gamma,\tau)} \hat{J}_{r}(\bar \gamma,\bar\tau)[\delta\gamma,\delta\tau]^2
+
\left\langle \bar\lambda,D^2 G(\bar\gamma)[\delta\gamma]^2\right\rangle
>0
\end{equation}
for every $\left(  \delta\gamma,\delta\tau\right)  \in H^{1}(0,T;\mathbb{R} ^{2})\times\mathbb{R} $, $\left(  \delta\gamma,\delta\tau\right)  \neq(0,0)$, belonging to the
critical cone
\begin{equation*}
	C(\bar{\gamma},\bar{\tau})
	=
	\set*{
		\begin{aligned}
		(\delta\gamma, \delta\tau) \in H^1(0,T;\R^2) \times \TT_{[3 r, T]}(\bar\tau);\;
		&
		g'(\bar\gamma)\delta\gamma \in \TT_K(g(\bar\gamma)),
		\\
		&
		\hat J_r'(\bar\gamma,\bar\tau)(\delta\gamma,\delta\tau) = 0
		\end{aligned}
	}
\end{equation*}
Then, $\hat{J}_{r}$ satisfies the quadratic growth condition
\begin{equation}
	\label{eq:qgcd}
	\hat J_r(\gamma, \tau)
	\ge
	\hat J_r(\bar\gamma, \bar\tau)
	+
	\frac\alpha2 \parens*{
		\norm*{ \gamma - \bar\gamma}_{H^1(0,T;\R^2)}^2
		+
		\abs*{ \tau - \bar\tau }^2
	}
	\qquad
\end{equation}
for all
$(\gamma,\tau) \in ( \Uad \times [3 r, T]) \cap \bar B( (\bar\gamma, \bar\tau); \varepsilon)$
with some constants $\alpha,\varepsilon > 0$.
In particular, $\hat{J}_{r}$ has a strict
local minimum at $\left(  \bar{\gamma},\bar{\tau}\right)  $.
\end{theorem}

\begin{remark}
\begin{itemize}
	\item The condition $g'(\bar\gamma)\delta\gamma \in\mathcal{T}_{K}(g(\bar{\gamma}))$
		in the statement of Theorem \ref{condsuff} is equivalent to $\delta\gamma \in \TT_{\Uad}(\bar\gamma)$,
		by \eqref{acq} and Corollary \ref{constqua}. 
\item
	We consider
	the particular case that $\Gamma_S$ is a closed convex domain of class $C^3$ in $\mathbb{R}^2$.
	W.l.o.g.\ we assume that $0 \in \Gamma_S$.
	Then, the Minkowski functional $M$ of is of class $C^3$ on $\mathbb{R}^2 \setminus \set{0}$,
	see \cite[Cor. 6.3.13, p.216]{KraPar}.
	Consequently,
	the function $g : x \mapsto \max(M(x) - 1/2, 0)^{4} - 1/2^{4}$
	is convex and of class $C^3$ on $\R^2$.
	Further, the properties of the Minkowski functional imply $\Gamma_S=\{x \in \mathbb{R}^2; g(x) \leq 0\}$
and $\nabla g(x) \not= 0$ for $x \in \partial\Gamma_S$.
In that case, the convexity of $g$ gives
\begin{equation}
\left\langle \bar\lambda,D^2 G(\bar\gamma)[\delta\gamma]^2\right\rangle
=
\int_{[0,T]}
\delta\gamma(t)^\top \nabla^2 g(\bar\gamma(t)) \delta\gamma(t)
\d\bar{\lambda}(t)\geq 0.
\end{equation}
Consequently,
if
$D^2_{(\gamma,\tau)} \hat{J}_{r}(\bar \gamma,\bar\tau)[\delta\gamma,\delta\tau]^2 > 0$  for every $(\delta\gamma,\delta\tau) \not= (0,0)$
belonging to the critical cone $C(  \bar{\gamma},\bar{\tau})$
then \eqref{prdumin}
is automatically satisfied.
Hence,
Theorem~\ref{condsuff}
improves
Corollary 8.12 in \cite{HNPW}
in
the case that $\Gamma_S$ is a closed regular convex domain of class $C^3$ in $\mathbb{R}^2$.
\end{itemize}
\end{remark}

By Theorem \ref{theo:CSN}, we also have the following necessary condition
with no-gap with respect to the sufficient condition.
\begin{theorem}
Let $(\bar\gamma,\bar\tau)$ be a local minimizer of our minimization problem \ref{ourminpb}. 
Then,
there exists a unique Lagrange multiplier $(\bar\lambda, \bar\mu) \in L^2(0,T) \times \R$
such that
$\bar\lambda \in \NN_K(g(\bar\gamma))$, $\bar\mu \in \NN_{[3 r, T]}(\bar \tau)$
and
\begin{equation}
	\label{eq:FONC}
	D_\gamma \hat J_r(\bar\gamma, \bar\tau)
	+
	\bar\lambda g'(\bar\gamma)
	=
	0,
	\quad\text{and}\quad
	D_\tau \hat J_r(\bar\gamma, \bar\tau)
	+ \bar\mu
	= 0
	.
\end{equation}
Further,
\begin{equation}
D^2_{(\gamma,\tau)} L(\bar \gamma,\bar\tau,\bar\lambda,\bar\mu)[\delta\gamma,\delta\tau]^2 \geq 0
\end{equation}
for every $\left(  \delta\gamma,\delta\tau\right)  \in H^{1}(0,T;\mathbb{R}^{2})\times\mathbb{R}
$ belonging to the critical cone $C\left(  \bar{\gamma},\bar{\tau}\right)$.
\end{theorem} 
\begin{proof}
	The first-order condition \eqref{eq:FONC}
	follows from Theorem~\ref{exLagMult}
	and Corollary~\ref{constqua},
	see also \eqref{31dcq}.
	The uniqueness of the Lagrange multiplier $(\bar\lambda,\bar\mu)$ follows from \eqref{31gcq} and $\bar\lambda \in L^2(0,T)$ 
	was shown in Proposition \ref{prop:lmsqi}.
	The second-order condition is a consequence of
	Theorem \ref{theo:CSN}.
\end{proof}

\section{\textbf{Appendix}}
\label{sec:regularity_stationary}

We have proved in section \ref{sec:regularity}, that if $(\bar\gamma, \bar\tau)$ is a local minimizer of our minimization Problem \ref{ourminpb}
that $\bar{\gamma} \in H^2(0,T;\mathbb{R}^2)$ and $\bar{\gamma}'(0)=\bar{\gamma}'(T)=0_{\mathbb{R}^2}$.
Is this still true for stationary points $(\bar\gamma, \bar\tau)$ in the sense of Definition \ref{statpoint} of our minimization Problem \ref{ourminpb}?

The following lemma shows
that
every stationary point $(\bar\gamma, \bar\tau)$ of our minimization Problem \ref{ourminpb}
is
a local minimizer of
the modified reduced cost functional
\begin{equation} \label{MRF}
	\tilde J_M(\gamma,\tau)
	:=
	\hat J_r(\gamma,\tau)
	+
	\frac{M}{2} \parens*{
		\norm{\gamma - \bar\gamma}_{H^1(0,T; \R^2)}^2
		+
		\abs{\tau - \bar\tau}^2
	}
\end{equation}
if $M > 0$ is chosen large enough.

\begin{lemma}
	\label{lem:stat_mod}
	Let $(\bar\gamma,\bar\tau)$ be a stationary point in the sense of Definition \ref{statpoint} of our minimization Problem \ref{ourminpb}.
	Then, for $M>0$ large enough, there exist constants $\varepsilon,\alpha > 0$
	such that the quadratic growth condition
	\begin{equation*}
		\tilde J_M(\gamma, \tau)
		-
		\tilde J_M(\bar\gamma, \bar\tau)
		\ge
		\frac{\alpha}{2} \parens*{
			\norm{\gamma - \bar\gamma}_{H^1(0,T; \R^2)}^2
			+
			\abs{\tau - \bar\tau}^2
		}
	\end{equation*}
	for all
	$(\gamma,\tau) \in \parens{\Uad\times [3r,T]} \cap B_\varepsilon \big((\bar\gamma,\bar\tau)\big)$
	is satisfied.
\end{lemma}
This lemma is a straightforward application of Theorem \ref{condsuff},
since the additional term in \eqref{MRF}
implies that the Hessian of the associated Lagrangian is coercive for $M > 0$ large enough.

Now, we can try to repeat the arguments from the proof of Theorem \ref{regoptraj}.
For a regularization parameter $\kappa > 0$, we define the regularized and modified reduced cost functional
\begin{equation*}
	\tilde J_{M,\kappa}(\gamma, \tau)
	:=
	\tilde J_{M}(\gamma, \tau)
	+
	\kappa \int_0^T [g\circ\gamma]_{+}^2(t)  \d t.
\end{equation*}
From $D_{\gamma}\tilde{J}_{M,\kappa}(\tilde\gamma_{\kappa},\tilde\tau_{\kappa})\delta\gamma=0$, for every $\delta\gamma \in H^1(0,T;\mathbb{R}^2)$,
we arrive by \eqref{add17}, \eqref{add20}, \eqref{add21} at the weak formulation of the ODE
\begin{equation} \label{ODEwithM}
	P_\kappa
	- (\lambda_\gamma \tilde\gamma_\kappa + M (\tilde\gamma_\kappa - \bar\gamma))''
	+ M (\tilde\gamma_\kappa - \bar\gamma)
	+ 2\kappa\left[  g\circ\tilde{\gamma}_{\kappa}\right]  _{+}(\nabla g)\circ\tilde{\gamma}_{\kappa} = 0
	.
\end{equation}
Thus
$\lambda_\gamma \tilde\gamma_\kappa + M (\tilde\gamma_\kappa - \bar\gamma) \in H^2(0,T;\R^2).$

However,
we were not able
to deduce from 
$(\lambda_\gamma \gamma_\kappa + M (\gamma_\kappa - \bar\gamma))'' \in L^2(0,T;\R^2)$ that $\bar\gamma$ itself belongs to $H^2(0,T;\R^2)$ by using the weak formulation of the ODE \eqref{ODEwithM} like in the proof of Theorem \ref{regoptraj}.

The important difference between the ODE \eqref{ODEwithM} and the ODE \eqref{add22}
is the additional term $-M(\tilde{\gamma}_{\kappa} - \bar\gamma)''$.
This term is due to the fact that in the modified reduced cost functional \eqref{MRF},
we have added $\frac{M}{2}\lVert \gamma - \bar\gamma \rVert _{H^1(0,T;\mathbb{R}^2)}^2$
instead of $\frac{M}{2}\lVert \gamma - \bar\gamma \rVert _{L^2(0,T;\mathbb{R}^2)}^2$ like in \eqref{add2}.
However,
but this difference is essential in order to get the coercivity of the Hessian of the Lagrangian
in the proof of Lemma~\ref{lem:stat_mod} for $M$ large enough.  

Thus, the strategy of the modified reduced cost functional does not seem to be able to give a positive answer to our question on the regularity of the trajectory in the case of stationary points for our minimization Problem  \ref{ourminpb}.

This regularity question remains an open problem.

\end{document}